\documentclass{amsart}

\usepackage{hyperref, cleveref}
\usepackage{graphicx}
\usepackage{mathtools, amssymb}
\usepackage{accents}
\usepackage{makecell}
\usepackage{multicol}
\usepackage{blkarray}
\usepackage{nicematrix}
\usepackage{bm}
\usepackage{paralist}
\usepackage[T1]{fontenc}
\usepackage{subcaption}
\usepackage[margin=1in]{geometry}
\usepackage[numbers,sort]{natbib}
\usepackage{thmtools}
\usepackage{xcolor}
\usepackage{tikz}
\usepackage{tikz-cd}

\usetikzlibrary{decorations.pathreplacing}
\usetikzlibrary{patterns,external}
\usetikzlibrary{calc}
\usetikzlibrary{shapes.multipart}
\usetikzlibrary{positioning}

\definecolor{fondo}{rgb}{0.898,0.996,0.898}
\pgfkeys{/tikz/.cd, K/.store in=\K, K=1}

\usepackage{enumitem}
\setlist[enumerate]{leftmargin=.5in}
\setlist[itemize]{leftmargin=.5in}

\usepackage{algorithm}
\usepackage{algpseudocode}
\algrenewcommand\algorithmicrequire{\textbf{Inputs:}}
\algrenewcommand\algorithmicensure{\textbf{Output:}}
\newcommand{\Bullet}[1]{\Statex \hspace{1.5em} $\bullet$ #1}
\usepackage{verbatim}

\newcommand{\Abullet}{{\mathcal{A}_{\bullet}}}
\newcommand{\actson}{\curvearrowright}
\newcommand{\C}{{\mathbb{C}}}
\newcommand{\MV}{\text{MV}}
\newcommand{\mydef}[1]{{\color{blue}#1}}
\newcommand{\mydefit}[1]{{\color{blue}\emph{#1}}}
\newcommand{\N}{{\mathbb{N}}}

\newcommand{\V}{{\mathcal{V}}}
\newcommand{\Z}{{\mathbb{Z}}}

\makeatletter
\DeclareRobustCommand{\sqcdot}{\mathbin{\mathpalette\morphic@sqcdot\relax}}
\newcommand{\morphic@sqcdot}[2]{%
  \sbox\z@{$\m@th#1\centerdot$}%
  \ht\z@=.33333\ht\z@
  \vcenter{\box\z@}%
}
\makeatother

\makeatletter
\def\fooC#1{%
\expandafter\newcommand\csname c#1\endcsname{\mathcal{#1}}}
\def\fooB#1{%
\expandafter\newcommand\csname b#1\endcsname{\mathbb{#1}}}
\count@=0
\loop
\advance\count@ 1
\edef\y{\@Alph\count@}%
\expandafter\fooC\y
\expandafter\fooB\y
\ifnum\count@<26
\repeat
\makeatother

\newtheorem{theorem}{Theorem}[section]
\newtheorem{definition}[theorem]{Definition}
\newtheorem{lemma}[theorem]{Lemma} 
\newtheorem{corollary}[theorem]{Corollary} 
 
\newtheorem{proposition}[theorem]{Proposition}

\theoremstyle{definition}
\newtheorem{example}[theorem]{Example}
\newtheorem{remark}[theorem]{Remark}

\title{Restricted monodromy containing transpositions}

\author[J. Barnhart]{Julianne Barnhart}
    \address[J. Barnhart, prev. at Clemson University]{Department of Mathematics and Computer Science, Whitworth University, Spokane, WA 99251}
    \email{jbarnhart@whitworth.edu}
\author[T. Brysiewicz]{Taylor Brysiewicz}
    \address[T. Brysiewicz]{Department of Mathematics, Western University, London, ON N6A 3K7, Canada}
    \thanks{Brysiewicz was supported by NSERC discovery grant RGPIN-2023-03551.}
    \email{tbrysiew@uwo.ca}
\author[M.A. Burr]{Michael Burr}
    \address[M.A. Burr]{School of Mathematical and Statistical Sciences, Clemson University, Clemson, SC, 29634}
    \thanks{Burr was supported by SIMONS travel support for mathematicians \#964285.}
    \email{burr2@clemson.edu}
\author[T. Yahl]{Thomas Yahl}
    \address[T. Yahl]{Department of Mathematics, University of Wisconsin, Madison, WI 53706}
    \email{tyahl@wisc.edu}

\begin{document}

\begin{abstract}
We relate the monodromy group of a branched cover to the monodromy groups of its restrictions. We show that the presence of a transposition in a restricted monodromy group significantly constrains the monodromy group of the original branched cover. Using this result, we complete the sparse trace test algorithm, which numerically verifies whether a set of solutions to a sparse polynomial system is complete. Along the way, we develop a generalization of the notion of a trace test in the setting of branched covers.
\end{abstract}


\maketitle

\section{Introduction}
The goal of this paper is to relate the monodromy groups of branched covers to the monodromy groups of their restrictions. A branched cover is a dominant map \[\pi: X \to Y
\] of irreducible varieties of the same dimension.
We identify its monodromy group $G_{\pi}$  with a subgroup of the symmetric group $S_d$ by choosing a generic fibre of $\pi$ and a labeling of its $d$ points.

For an irreducible subvariety $V \subseteq Y$, we examine the restriction of $\pi$ to $\pi^{-1}(V)$, denoted by
\[
\pi_V: \pi^{-1}(V) \to V.
\]
When this restriction is a branched cover of the same degree as $\pi$, there is a natural containment of monodromy groups $G_{\pi_V} \leq G_\pi$.  In general, little more may be said about the relationship between $G_{\pi_V}$ and $G_\pi$. Thus, we focus on restricted branched covers over subvarieties $V$ which are \textit{generic level sets} of a fixed dominant map
\[
\psi: Y \to Z.
\]
Such restrictions are commonplace throughout the literature. For example, when $Y = \mathbb{C}^n$, a generic projection $\psi: Y \to \mathbb{C}^{n-1}$ has lines as level sets, and the monodromy group restricted to a \emph{generic} line of this form coincides with the original group $G_{\pi}$ \cite{Zar29,NumericalGalois18}. Bounds on the dimension of the subvariety of projections $\psi$ that preserve the monodromy group are also known \cite{Poon20}. 

Our main result extends properties of $G_{\pi_V}$ to $G_{\pi}$ when $V$ is a generic level set and $G_{\pi_V}$ contains a transposition. The list in the following theorem gives a blueprint for our paper.

\begin{theorem}
    \label{thm:main_theorem}
    Suppose $\pi: X \to Y$ is a branched cover of degree $d$, $\psi: Y \to Z$ is a dominant map with generic level set $V$, and $\pi_V$ is a branched cover of degree $d$ such that $G_{\pi_V}$ contains a transposition.  If $G_{\pi_V}$ is not the full symmetric group, then 
    \begin{enumerate}
        \item The orbits $\mathcal B_V$ of the subgroup generated by transpositions in  $G_{\pi_V}$ form a block system of~$G_{\pi_V}$,       
        \item $G_{\pi_V}$ is permutation-isomorphic to a full wreath product $S_k \wr H$ for some action  $H \actson \mathcal B_V$,
        \item $\mathcal B_V$ is the unique minimal block system of $G_{\pi_V}$, and
       \item  $\mathcal B_V$ is the unique minimal block system of $G_{\pi}$
    \end{enumerate}
    Therefore, $G_{\pi_V}$ is the full symmetric group if and only if $G_\pi$ is the full symmetric group.
\end{theorem}

Our main contribution is the fourth statement of \Cref{thm:main_theorem}, while the first three statements of the theorem consist of applications of elementary results in the present setting.  The first statement is purely group-theoretic: every permutation group $G$ preserves the orbit partition of any of its normal subgroups, and the subgroup $T(G)$ of $G$ generated by transpositions is normal in $G$. Since $G_{\pi_V}$ has a transposition but is not the full symmetric group, the orbits of $T\left(G_{\pi_V}\right)$ form a non-trivial partition $\mathcal B$ preserved by $G_{\pi_V}$, also known as a block system.  For the second statement, we note that groups preserving non-trivial partitions are not necessarily wreath products. However, in our case, since $T\left(G_{\pi_V}\right)$ is generated by transpositions, it is a product of symmetric groups on each of its orbits.  By the transitivity of $G_{\pi_V}$, each orbit has size $k$ for some $1<k<d$. We conclude that $G_{\pi_V}$ is permutation-isomorphic to a full wreath product $S_k \wr H$ where $H \actson \mathcal B$ is the induced action of $G_{\pi_V}$ on $\mathcal B$.  The third statement then follows directly from the previous statements.

The fourth statement says that, assuming the hypotheses of the theorem,  the block system $\mathcal B$ of $G_{\pi_V}$ is also a block system of  $G_\pi$. We prove this using the following approach: each generic point in $Y$ belongs to a generic level set $V$ of $\psi$, and thus the fibre over any generic point comes equipped with a minimal block system $\mathcal B_V$ by the first and third statements of the theorem. We prove, using Thom's isotopy lemma \cite{Thom1969}, that paths in a Euclidean neighborhood of $V$ which avoid the branch locus of $\pi$ respect these minimal block systems. Equipped with the uniqueness of $\mathcal B_V$, one may chain these neighborhoods together to cover a global monodromy loop of $\pi$, showing that the block system $\mathcal B_V$ is preserved globally as well. The local behavior of block systems over paths is illustrated in \Cref{fig:shematic_block_system}.  The final statement of \Cref{thm:main_theorem} is a corollary of the other statements and is the main tool in our applications.

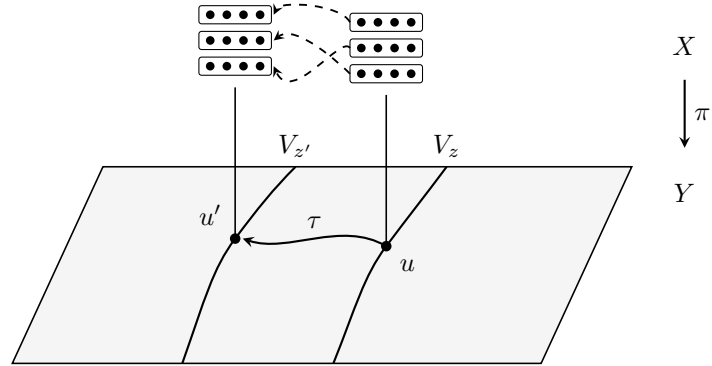
\begin{figure}[!htpb]
\begin{center}
\begin{tikzpicture}[
    scale=1,
    >=stealth,
    floor/.style={fill=gray!8, draw=black, line width=0.6pt},
    curve/.style={draw=black, line width=0.8pt},
    point/.style={circle, fill=black, inner sep=1.4pt},
    lift/.style={draw=black, line width=0.6pt},
    box/.style={draw=black, rounded corners=1pt, inner sep=2pt},
    tau/.style={draw=black, line width=0.8pt, ->},
    fiberpath/.style={draw=black, dashed,line width=0.7pt, ->}
]
\coordinate (A) at (0,0);
\coordinate (B) at (7,0);
\coordinate (C) at (8.2,2.6);
\coordinate (D) at (1.2,2.6);
\draw[floor] (A) -- (B) -- (C) -- (D) -- cycle;
\node at (8.9,2.25) {$Y$};
\node at (8.9,4.2) {$X$};
\draw[curve]
    (2.25,0)
    .. controls (2.55,0.85) and (2.65,1.25) .. (2.95,1.65)
    .. controls (3.25,2.05) and (3.45,2.30) .. (3.75,2.6)
    node[above] {$V_{z'}$};
\draw[curve]
    (4.25,0)
    .. controls (4.55,0.75) and (4.65,1.15) .. (4.95,1.55)
    .. controls (5.25,1.95) and (5.45,2.20) .. (5.75,2.6)
    node[above] {$V_z$};
\coordinate (u)  at (4.95,1.55);
\coordinate (up) at (2.95,1.65);
\coordinate (upR) at (3.05,1.65);
\node[point, label=below right:{$u$}] at (u) {};
\node[point, label=above left:{$u'$}] at (up) {};
\draw[->, line width=0.8pt]
    (8.9,3.75) -- node[right] {$\pi$} (8.9,2.85);
\draw[tau]
    (u) .. controls (4.35,1.90) and (3.65,1.45) ..
    node[above] {$\tau$} (upR);
\coordinate (Utop)  at (4.95,3.55);
\coordinate (UPtop) at (2.95,3.65);
\draw[lift] (u) -- (Utop);
\draw[lift] (up) -- (UPtop);
\foreach \r in {0,1,2} {
    \draw[box]
        ($(Utop)+(-0.48,0.16+0.34*\r)$)
        rectangle
        ($(Utop)+(0.48,0.40+0.34*\r)$);
    \foreach \c in {0,1,2,3} {
        \node[point, inner sep=1.15pt]
            at ($(Utop)+(-0.33+0.22*\c,0.28+0.34*\r)$) {};
    }
}
\foreach \r in {0,1,2} {
    \draw[box]
        ($(UPtop)+(-0.48,0.16+0.34*\r)$)
        rectangle
        ($(UPtop)+(0.48,0.40+0.34*\r)$);
    \foreach \c in {0,1,2,3} {
        \node[point, inner sep=1.15pt]
            at ($(UPtop)+(-0.33+0.22*\c,0.28+0.34*\r)$) {};
    }
}
\draw[fiberpath]
    ($(Utop)+(-0.48,0.28)$)
    .. controls (4.35,3.95) and (3.75,4.5) ..
    ($(UPtop)+(0.51,0.62)$);

\draw[fiberpath]
    ($(Utop)+(-0.48,0.62)$)
    .. controls (4.35,4.35) and (3.7,3.40) ..
    ($(UPtop)+(0.51,0.28)$);

\draw[fiberpath]
    ($(Utop)+(-0.48,0.96)$)
    .. controls (4.35,4.75) and (3.55,4.80) ..
    ($(UPtop)+(0.51,0.96)$);
\end{tikzpicture}
\end{center}
\caption{An illustration of two generic points $u,u' \in Y$ and their respective generic level sets with respect to a map $\psi:Y\rightarrow Z$ (omitted from the diagram). Each generic point of $Y$ has twelve points in its fibre, which are partitioned into three quadruples giving a minimal block system. The twelve lifts of a path $\tau$ from $u$ to $u'$ with respect to $\pi$ respect this partition.}\label{fig:shematic_block_system}
\end{figure}

\subsection{Application to system solving}
Our main application and motivation for this work is computational. In \cite{STT23}, the second and third authors developed the \textit{sparse trace test}, which generalizes the \textit{(classical) trace test} \cite{Som02,Ley18} to zero-dimensional systems. Trace tests are essentially \textit{completion tests} for solution sets to parametrized polynomial systems. Thus, they are indispensable subroutines for solvers that build such solution sets incrementally (see, for example, \cite{DuffMonodromy}). A trace test checks completion by showing that the local extension of a \textit{trace function}, given by analytic continuation over the parameter space,  extends to a global function of the parameters.

In the classical setting, the polynomial system of interest represents zero-dimensional slices of an irreducible variety, where the slices themselves form the parameter space. The trace function is the coordinate-wise sum of the solutions. It is a function if and only if it is evaluated on the complete solution set, in which case it is known to be a linear function. The trace test therefore checks whether the coordinate-wise sum is a function by assessing this linearity numerically. The argument that incompleteness implies nonlinearity relies on the ability to swap a found-solution with a solution that has not been found by moving the affine linear slice. This works because the corresponding monodromy group is known to be the full symmetric group. 

In the sparse trace test, the trace function is the sum of a single coordinate of the solutions. In the current paper, functionality is assessed with respect to  moving a single coefficient in the defining equations, as illustrated in \Cref{fig:applications}(a). Indeed, moving only one point is a significant improvement over \cite{STT23}, where moving a subset of points is required.  This  family of polynomial systems is interpreted as a level set of a map which forgets only one coefficient of a polynomial system. By Khovanskii's result on the irreducibility of underdetermined sparse polynomial systems, when the level set is generic, the restricted map is indeed a branched cover \cite{Khov16}. Esterov's results pertaining to discriminants of sparse polynomial systems identify the many cases where these restricted monodromy groups contain a transposition \cite{Est19}. In this case, the machinery of our main theorem may be applied to show that the restricted monodromy group of the branched cover corresponding to moving a single coefficient is the full symmetric group. Thus, our work completes the converse of the sparse trace test.
\begin{figure}[!htpb]
    \centering
    \includegraphics[width=0.95\linewidth]{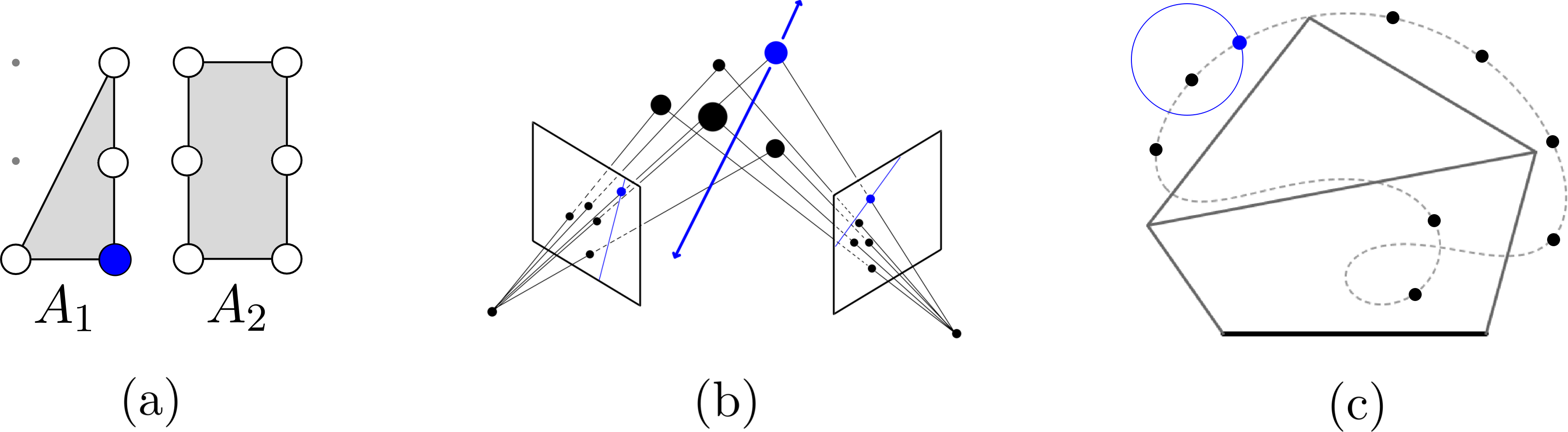}
    \caption{(a) Two supports representing a sparse polynomial system and their convex hulls. Varying the coefficient corresponding to the filled blue point represents restriction of the relevant branched cover to a level set. (b) The five-point minimal problem from algebraic vision. Varying the blue point along a line represents restriction to a level set. (c) The nine point synthesis problem. Varying the blue point along a circle represents the restriction to a level set.}
    \label{fig:applications}
\end{figure}

 A major contribution of this work is the development of the notion of a trace test in the general setting of branched covers, achieved in \Cref{sec:generalizedtrace}. This generalization is phrased in terms of whether a certain function on the universal cover of the base of a branched cover descends to a function on the base itself. It specializes to both the classical and sparse trace tests. Additionally, we formalize the aforementioned property of being able to \textit{swap} a found solution for a non-found solution group-theoretically and conjecturally provide a census of all permutation groups which have this property other than the natural alternating and symmetric groups.

\subsection{Application to monodromy groups}
Beyond the sparse trace test, our work is relevant to popular polynomial systems 
coming from applications such as algebraic vision 
\cite{DKPR22}, kinematics 
\cite{NumericalGalois18,DKPR23}, likelihood estimation 
\cite{NumericalGalois18}, power flow 
\cite{MNMH16},  discrete gravity 
\cite{AB24}, and rigidity theory 
\cite{MSW26}. In these settings, branched covers are often encoded via parametrized polynomial systems, and restrictions correspond to natural real-world constraints on the relevant parameter spaces. \Cref{fig:applications} displays two such manifestations of branched cover restrictions. \Cref{fig:applications}(b) represents the $5$ point minimal problem, where one seeks to reconstruct five points in three-dimensional space from two camera images, along with the relationship between those cameras. The five points form the parameter space, and the movement of one point along a line represents the restriction to a level set.  \Cref{fig:applications}(c) shows the nine-point synthesis problem, which is the planar interpolation problem for coupler curves of four bar mechanisms. Our image shows the restriction of this problem to a generic set of nine points, where one point is allowed to move freely along a circle.  Our results reinforce the structures of the monodromy groups computed for these types of systems in previous work.  For example, we explain why the monodromy groups computed in \cite[Results 3.1 and 4.1]{DKPR22} and \cite{AB24,MNMH16} are members of one of the following classes: full symmetric groups, full wreath products, or permutation groups without transpositions.

\section{Background on permutation groups} For any finite set $\Omega=\{x_1,\ldots,x_d\}$, we write $\mydef{S_\Omega}$ for the \mydefit{symmetric group on $\Omega$}, which consists of bijections from $\Omega$ to itself. An ordering of $\Omega$ identifies $S_\Omega$ with the more familiar symmetric group $\mydef{S_d}\vcentcolon=S_{[d]}$  on $\mydef{[d]}\vcentcolon= \{1,2,\ldots,d\}$. Any subgroup $G \leq S_\Omega$ is called a \mydefit{permutation group} and naturally inherits a group action on $\Omega$ from $S_\Omega$. 
The orbits $\Omega_1,\ldots,\Omega_m$ of this action partition $\Omega$. If $m=1$, then $G$ is called \mydefit{transitive}, and otherwise $G$ is called \mydefit{intransitive}. 

\subsection{Imprimitivity and wreath products} Let $\mathcal B=\{B_1,\ldots,B_b\}$ be a set partition of $\Omega$. We say $G$ \mydefit{preserves}  $\mathcal B$  if for all $g \in G$ and $B \in \mathcal B$
\begin{equation}
    \label{eq:preserves_partition}
\mydef{g(B)} \vcentcolon=\{g(b) \mid b \in B\}= B' \text{ for some }B' \in \mathcal B.
\end{equation}
Every permutation group preserves the finest and the coarsest partitions, called the \mydefit{trivial partitions}. Similarly, every permutation group preserves its orbit partition. 
If a \emph{non-trivial partition} is preserved by a \textit{transitive group} $G$, then the partition is called a \mydefit{block system} and the parts are called \mydefit{blocks}.  The blocks of a block system all have the same size. 
If a transitive group $G$ preserves a block system $\mathcal B$, then $G$ is called \mydefit{imprimitive}, and if no such block system exists, then $G$ is called \mydefit{primitive}.   We say that  $\cB$ is \mydefit{minimal} if $G$ preserves no non-trivial refinement of $\mathcal B$. 

Given two permutation groups $Q \leq S_k$ and $H \leq S_b$, the \mydefit{wreath product} of $Q$ by $H$, denoted by $Q \wr H$, is the semidirect product $Q \wr H = Q^b \rtimes H$ where $H$ acts on the factors of $Q\times\cdots\times Q$ by permuting indices. Such a wreath product is naturally realized as a permutation subgroup of $S_{kb}$ as the join of the intransitive action 
\[\underbrace{Q \times \cdots \times Q}_{b} \leq \underbrace{S_k \times \cdots \times S_k}_{b} \leq S_{kb}.
\]
That is, an element $(q_1,\ldots,q_b;h)$ acts on $[k]\times[b]$ as $(i,j) \mapsto (q_{h(j)}(i),h(j))$. If $Q=S_k$ is the full symmetric group, we say that $Q \wr H$ is the \mydefit{full wreath product over $H$}. 

Given a partition $\mathcal{B}=\{B_1,\dots,B_b\}$ of a set $\Omega$, the permutations which preserve $\mathcal B$ comprise a group denoted by $\mydef{\textrm{Wr}(\mathcal B)}$.  If $\mathcal B$ is a block system of some $G \leq S_{\Omega}$, then $G$ is a transitive subgroup of $\textrm{Wr}(\mathcal B)$, every $B_i$ has the same size $k$, and  $\textrm{Wr}(\mathcal B)$ is permutation-isomorphic (that is, equal after relabeling $\Omega$) to the representation of $S_k\wr S_b$ given above.  Moreover, associated with $G$ are the following groups:
\begin{itemize}
\item $H\leq S_{\mathcal{B}}$, the induced action of $G$ on the set of blocks and
\item $K\leq\prod_{i=1}^b S_{B_i}$, the kernel of that action.
\end{itemize}
In particular, $K$ is normal in $G$ and $G/K\cong H$. 
Since $G$ is transitive, so is $H$. There are three natural associated actions on a single block $B_i$:
\begin{itemize}
\item $Q'_i$, the induced action of the pointwise stabilizer of the complement of $B_i$,
\item $Q''_i$, the induced action of $K$ on $B_i$, and 
\item $Q'''_i$, the induced action of $\textrm{Stab}_G(B_i)$ on $B_i$.
\end{itemize}
We observe that $Q'_i\leq Q''_i\leq Q'''_i$ and that 
$$\prod_{i=1}^b Q'_i\leq K\leq \prod_{i=1}^b Q''_i\leq \prod_{i=1}^b Q'''_i.$$
Since $G$ is transitive, all of the $Q'_i$, $Q''_i$, or $Q'''_i$ are permutation-isomorphic to a common $Q'$, $Q''$, or $Q'''$, respectively. One may take $Q',Q'',Q''' \leq S_k$ and label  the elements of each $B_i$ by $[k]$ so that each $Q'_i$, $Q''_i$, or $Q'''_i$ is equal to the common permutation group $Q'$, $Q''$, or $Q'''$, respectively.  
The following example shows that the inclusions $Q'\leq Q''\leq Q'''$ may be proper. However, if they all coincide with some group $Q$, then $K=Q^b$, and $G$ is permutation-isomorphic to the wreath product $Q^b\rtimes H=Q\wr H$ in its natural permutation representation.

\begin{example}
When $G$ is an imprimitive group, but not a wreath product, it is possible that $Q'$, $Q''$, and $Q'''$ are distinct.  For example, when $G\leq S_9$ is generated by $(1,2,3)(4,5,6)(7,8,9)$ and $(1,4,7)(2,5,8)(3,6,9)$, then $Q'$ is trivial while $Q''$ is the cyclic group on three elements.  On the other hand, when $G\leq S_8$ is generated by $(1,3,5,7)(2,4,6,8)$ and $(1,2)(3,8)(4,7)(5,6)$, then $Q''$ is trivial while $Q'''$ is the symmetric group on two elements.
\end{example}

\subsection{Groups with transpositions}
 We define the \mydefit{transpositional subgroup} of a permutation group $G$ to be the group $T(G) \leq G$ generated by all transpositions in $G$. If there are no transpositions in $G$, $T(G)$ is the trivial subgroup. Since conjugation preserves transpositions, the subgroup $T(G)$ is normal in $G$, so $G$ preserves the orbits of $T(G)$. 
\begin{lemma}\label{lem:YoungBlocks}
    Fix $G \leq S_{\Omega}$ and let $\Omega_1,\ldots,\Omega_m$ be the orbits of $T(G)$, then 
    \[
    T(G) = S_{\Omega_1}\times S_{\Omega_2} \times \cdots \times S_{\Omega_m}
    \]
    and $G$ preserves the orbit partition $\{\Omega_1, \ldots, \Omega_m\}$ of $T(G)$.
\end{lemma}
\begin{proof} 
    The transpositional subgroup $T(G)$ acts transitively on each orbit $\Omega_i$. Each transitive action $T(G) \actson \Omega_i$ is generated by transpositions and therefore must be the symmetric group $S_{\Omega_i}$. On the other hand, $T(G) \leq \prod_{i=1}^m S_{\Omega_i}$, and so they must be equal.

    To see that $G$ preserves the orbit partition of $T(G)$, fix $g \in G$. Suppose that $i$ and $j$ belong to the same orbit  of $T(G)$, or equivalently, since $T(G)$ is the product of symmetric groups, $(i,j)\in T(G)$. Then $g(i,j)g^{-1}$ is the transposition $(g(i),g(j))$, which must be in $T(G)$. This proves $g(i)$ and $g(j)$ belong to the same orbit of $T(G)$.  The converse holds by replacing $g$ with $g^{-1}$.
\end{proof}

The following two corollaries of \Cref{lem:YoungBlocks} prove the first three parts of \Cref{thm:main_theorem}.
\begin{corollary}\label{cor:uniformorimprimitive}
Let $G \leq S_{\Omega}$ be a transitive permutation group, and let $\mathcal B=\{\Omega_1,\ldots,\Omega_m\}$ be the orbit partition of  $T(G)$. Then $G$ is the full wreath product $G \cong S_k \wr H$ over the action $H$  of $G$ on $\mathcal B$ and $k=|\Omega|/m=|\Omega_i|$. 
In particular, if $G$ contains a transposition, then either 
\begin{enumerate}
    \item $k=|\Omega|$, $m=1$, and $G$ is the full symmetric group, or 
\item  $k<|\Omega|$ and $G$ is imprimitive with block system $\mathcal B$. 
\end{enumerate}
\end{corollary}
\begin{proof}
    By \Cref{lem:YoungBlocks}, the orbits $\mathcal B = \{\Omega_1,\ldots,\Omega_m\}$ of the transpositional subgroup $T(G)$ are preserved by $G$, so $G \leq \textrm{Wr}(\mathcal B)$. Since $G$ is transitive, it induces a transitive action $H \leq S_m$ on the orbits with kernel $K \leq \prod_{i=1}^{m} S_{\Omega_i}$. Since $\prod_{i=1}^m S_{\Omega_i} = T(G) \leq K$, the subgroup containments collapse and $K=T(G)$.  In particular, we observe that for any $i$, $G$ contains elements that fix all elements not in $\Omega_i$ and applies any permutation to $\Omega_i$.  Hence, from the discussion of wreath products, $Q'=S_{\Omega_i}$.  Since $Q'$ is as large as possible, it follows that the containments of $Q'\leq K\leq Q''\leq Q'''$ collapse, $G/T(G) = H$, and $G = T(G)\rtimes H\cong S_k \wr H$.
    
    If $G$ contains no transpositions, then $k=1$, $H=G$, and $G$ is the full (trivial) wreath product $S_1 \wr G$. If $k=|\Omega|$, then $m=1$ and $G \cong S_{\Omega} \wr \{e\} = S_{\Omega}$. If $G$ contains a transposition but is not the full symmetric group, then \Cref{lem:YoungBlocks} implies that $G$ is imprimitive with a non-trivial block system given by $\mathcal B$.
\end{proof}

\begin{corollary}\label{cor:uniqueness}Let $G \leq S_{\Omega}$ be an imprimitive group with block system $\mathcal B=\{B_1,\dots,B_b\}$ such that
\[
G=\left(\prod_{i=1}^b S_{B_i}\right)\rtimes H,
\]
where $H$ is the action of $G$ on the blocks of $\mathcal B$. Then $\mathcal B$ is the unique minimal block system of $G$.
 
\end{corollary}
\begin{proof}
Suppose that $\mathcal B' \neq \mathcal B$ is a minimal block system of $G$. Then there exist $i,j \in \Omega$ such that $i$ and $j$ belong to the same block $B \in \mathcal B$, but belong to distinct blocks $B'_i$ and $B'_j$ of $\mathcal B'$, respectively. Indeed, if no such pair existed, then every block of $\mathcal B$ would be contained in a block of $\mathcal B'$, so $\mathcal B$ would refine $\mathcal B'$, contradicting the minimality of $\mathcal B'$.

Since $\mathcal B'$ is a block system, its blocks are not singletons. Hence, there exists $i' \neq i$ such that $i' \in B'_i$.
Now $i,j \in B$ for some block $B \in \mathcal B$, and by hypothesis, $\{e\}^{b-1}\times S_B$ is contained in $G$.  In particular, the transposition $(i,j)$ belongs to $G$. This transposition fixes $i'$, while exchanging $i$ and $j$. Thus $(i,j)$ sends the block $B'_i$ to a set containing both $i'$ and $j$, but not $i$, so $(i,j)B'_i$ is neither equal to $B'_i$ nor disjoint from it. Therefore $(i,j)$ does not preserve the partition $\mathcal B'$, contradicting the assumption that $\mathcal B'$ is a block system for $G$.
\end{proof}

\section{Monodromy of Branched Covers}\label{sec:MonodromyBranched}
Our main object of study is a dominant map
\begin{equation}
    \label{eq:branched_cover}
    \pi:X \to Y,
\end{equation}
 of irreducible varieties of the same dimension. Such a map is called a \mydefit{branched cover}.  There exists a dense open and connected set $\mydef{\mathcal U}\vcentcolon=Y - \Delta$ of the base space $Y$ over which the map $\pi|_{\pi^{-1}(\cU)}$ is a topological covering space of degree $d$. Indeed, $\Delta$ may be taken to be a proper subvariety of $Y$, and is called the \mydefit{branch locus} of the branched cover.  For any $u \in \mathcal U$, we denote the fibre $\pi^{-1}(u)$ by $\mydef{X_u}$. By the theory of covering spaces, a path $\tau:[0,1] \to \mathcal U$ lifts to $d$ paths, each beginning at a distinct point $\{x_1,\ldots,x_d\}$ in the fibre  $X_{\tau(0)}$ and ending at distinct points in the fibre $X_{\tau(1)}$. Writing these paths as $x_i(t):[0,1] \to X$ gives a bijection
\begin{align*}
\sigma_{\tau}:X_{\tau(0)} &\to X_{\tau(1)} \\
 x_i(0) &\mapsto x_i(1).
\end{align*}
The map $\tau \mapsto \sigma_\tau$ is independent of endpoint-preserving homotopies.  That is, it is well-defined on the \mydefit{fundamental groupoid} $\mydef{\Pi_1(\mathcal U)}$ which, as a set, is comprised of endpoint-preserving homotopy classes of paths in $\cU$. The groupoid operation is concatenation of paths $(\gamma, \tau) \mapsto \gamma \sqcdot \tau$  (using the notation of \cite[p.26]{Hatcher}) which is defined if and only if the appropriate endpoints agree. 

We write $\mydef{\Pi_1(\mathcal U,u,u')}$ for the set of endpoint-preserving homotopy classes of paths from $u$ to $u'$.  From this, we construct the \mydefit{universal cover} of $\cU$, whose underlying set consists of all equivalence classes of paths starting at a fixed point $u$, and denote it by 
$$
\mydef{\Pi_1(\cU,u,\cdot)}\vcentcolon=\coprod_{u'\in\cU}\Pi_1(\cU,u,u').
$$
If the endpoints of a path $\gamma$ are equal, then we call $\gamma$ a \mydefit{loop} based at $u=\gamma(0)=\gamma(1)$. The classes of all loops in $\mathcal U$ based at $u$ is called the \mydefit{fundamental group} $\mydef{\Pi_1(\mathcal U,u)}\vcentcolon=\Pi_1(\mathcal U,u,u)$ of $\mathcal U$ based at $u$. The groupoid operation on $\Pi_1(\mathcal U)$ restricts to a group operation on $\Pi_1(\mathcal U,u)$.

The map $\gamma \mapsto \sigma_\gamma$ restricted to $\Pi_1(\mathcal U,u)$ is a well-defined group homomorphism to the symmetric group $S_{X_u} \cong S_d$, which we call the \mydefit{monodromy homomorphism}. The image of the monodromy homomorphism is the \mydefit{monodromy group of $\pi$ based at $u$}, and it is denoted by $\mydef{G_{\pi,u}}$. Since $\mathcal U$ is connected, the fundamental group of $\mathcal U$ is well-defined up to conjugation by an element in  the fundamental groupoid.  Specifically, for any $[\tau]\in\Pi_1(\cU,u,u')$, we have the isomorphism
\begin{align*}
\Phi_\tau:\Pi_1(\mathcal U,u) &\xrightarrow{\cong} \Pi_1(\mathcal U,u')\\
[\gamma] &\mapsto [\tau^{-1} \sqcdot \gamma \sqcdot \tau].
\end{align*}
Thus, any two monodromy groups $G_{\pi,u}$ and $G_{\pi,u'}$ are isomorphic via conjugation by $\sigma_\tau$.  Consequently, identifying any $G_{\pi,u} \leq S_{X_u}$ with a subgroup $G_{\pi} \leq S_d$ is well-defined up to conjugation in $S_d$ and this conjugacy class does not depend on the base point. We call  $G_{\pi}$ \mydefit{the monodromy group} of $\pi$. 

\begin{lemma}
    If $\pi:X \to Y$ is a branched cover, then $G_{\pi}$ is transitive. 
\end{lemma}
\begin{proof}
Since $\Delta$ is a proper subset of $Y$ and $\pi$ is dominant, $\pi^{-1}(\Delta)$ is a proper subvariety of $X$.  Since $X$ is irreducible, $\pi^{-1}(\Delta)$ is of real codimension at least $2$ in $X$.  Hence $\pi^{-1}(\mathcal U)=X- \pi^{-1}(\Delta)$ is path-connected.  Therefore, for $u \in \mathcal U$ there exists $\tau:[0,1] \to \pi^{-1}(\mathcal U)$ with endpoints $\tau(0)=x_i$ and $\tau(1)=x_j$ in $X_u$.  The projected path $\gamma = \pi \circ \tau$ is a loop for which $\sigma_{\gamma}(x_i)=x_j$. 
\end{proof}

\subsection{Decomposability and imprimitivity} A \mydefit{factorization} of a branched cover $\pi:X\rightarrow Y$ is a rewriting of $\pi$ as a composition of branched covers $\varphi_1$ and $\varphi_2$:
    \tikzstyle{longdashed}=[dash pattern=on 6pt off 3pt]
\begin{equation}
    \label{eq:decomposable}
    \begin{tikzcd}
  X \arrow[r,longdashed,"\varphi_1"] \arrow[rr,bend right=25,longdashed,"\pi" below] & X' \arrow[r,longdashed,"\varphi_2"] & Y
\end{tikzcd}
\end{equation}
where the dashed arrows indicate that the relevant spaces may be replaced by Zariski open subsets. A factorization is a \mydefit{decomposition} if the degrees $k$ and $b$ of $\varphi_1$ and $\varphi_2$, respectively, are both greater than one.  A branched cover is \mydefit{decomposable} if there exists a decomposition of $\pi$.  We say $\pi$ is \mydefit{indecomposable} if no such decomposition exists. The existence of a decomposition of a branched cover has group-theoretic implications for its monodromy group.
\begin{lemma}
    \label{lem:decomposable_preserves_blocks}
    Let $\pi:X\rightarrow Y$ be a branched cover. If $\pi$ decomposes as $\pi=\varphi_2\circ\varphi_1$, then for $u\in\cU$, $G_{\pi,u}$ preserves the block system 
    $\mathcal B_u\vcentcolon=\{\varphi_1^{-1}(x')\}_{x' \in \varphi_2^{-1}(u)}.
    $
\end{lemma}
\begin{proof}
    Fix $x'\in\varphi_2^{-1}(u)$ and set $B=\varphi_1^{-1}(x')=\{x_1,\ldots,x_k\}$.  Let $\gamma$ be a loop based at $u$ and $x'(t)$ be the lift of $\gamma$ via $\varphi_2$ such that $x'(0)=x'$. Then $\varphi_1$ lifts $x'(t)$ to the paths $B(t)=\{x_1(t),\dots,x_k(t)\}=\varphi_1^{-1}(x'(t))$, where $x_i(0)=x_i$.  In particular, $B(1) = \varphi_1^{-1}(x'(1))$, which is a part of $\mathcal B_u$ since $x'(1) \in \varphi_2^{-1}(u)$.  Hence, $\sigma_\gamma$ preserves the partition $\cB_u$.
\end{proof}

\Cref{lem:decomposable_preserves_blocks} shows that decomposable branched covers have imprimitive monodromy groups. The converse is also true.

\begin{lemma}[{cf. \cite[Proposition 1]{Bry21}}]
    \label{lem:imprimitive_iff_decomposable}
    A branched cover is decomposable if and only if its monodromy group is imprimitive.
\end{lemma}

\subsection{Restrictions}
\label{sec:restrictions}
We extend the notation for a fibre $X_u$ of $\pi$ over a point $u \in Y$ to the preimage of a subset $V \subseteq Y$ by writing $\mydef{X_V}\vcentcolon=\pi^{-1}(V)$. We write the restriction of $\pi$ to $\pi^{-1}(V)$ as
\[
\mydef{\pi_V}: X_V \to V.
\] 
When $V$ is an irreducible subvariety of $Y$ and $V\cap\cU\not=\emptyset$, then $\pi_V$ is a dominant degree-$d$ map. It is a branched cover if $X_V$ is also irreducible.  

All branched covers with interesting (that is, not full symmetric) monodromy groups appear as restrictions of branched covers with full symmetric monodromy groups. Thus, studying all restrictions of $\pi$ to subvarieties of $Y$ is too ambitious. Instead, we focus our attention on restrictions of branched covers to \textit{generic subvarieties}. To make this notion precise, consider a dominant map $\psi:Y \to Z$.  Level sets of the form $\mydef{V_z} \vcentcolon= \psi^{-1}(z)$, where $z$ is a generic point in $Z$, are called \mydefit{generic}.  If a generic level set of $\psi$ produces a branched cover $\pi_V$, then we say that  $\psi$  is \mydefit{$\pi$-generic}.

Given a $\pi$-generic map $\psi:Y \to Z$ and a point $u \in \mathcal U$, the inclusion $\iota:V_{\psi(u)}-\Delta\hookrightarrow \cU$ induces a map of fundamental groups
\[
\iota_\ast:\Pi_1(V_{\psi(u)} - \Delta, u ) \rightarrow \Pi_1(\mathcal U,u).
\]
Every loop in $V_{\psi(u)}-\Delta$ is also a loop in $\cU$, and so this map, although it may not be injective, induces an inclusion of $G_{\pi_{V_{\psi(u)}}}$ as a subgroup of $G_{\pi}$.

Our goal is to relate these two groups when $u$ is generic.  Note that if $u$ is generic, then $w=\psi(u)$ is also generic.  Our first result shows that for generic $w \in Z$, one may thicken the level set $\psi^{-1}(w)$ to $\mydef{V_W} \vcentcolon=\psi^{-1}(W)$, where $W \subseteq Z$ is a Euclidean neighborhood of $w$, without enlarging the fundamental group of the complement of $\Delta$. The proof applies Thom's isotopy lemma \cite{Thom1969}.

\begin{lemma}
    \label{lem:generic_strip_fundamental_equivalence}
    Let $\psi: Y \to Z$ be a $\pi$-generic dominant map and let $u \in \mathcal U$ be a generic point. There exists a Euclidean neighborhood $W$ of $w= \psi(u)$ for which the inclusion $\iota:V_w \hookrightarrow V_W$ induces a surjection on fundamental groups, that is
    \[ \iota_\ast:\Pi_1(V_{w}- \Delta,u) \twoheadrightarrow 
    \Pi_1(V_W- \Delta,u) .
    \]
\end{lemma}
\begin{proof} Throughout this proof, we use the Euclidean topology. When discussing an arbitrary variety, we assume that it is embedded in projective space, an affine chart in $\mathbb{C}^n$ is taken, and we use the Euclidean subset topology.  We may also assume that $\psi$ is a polynomial dominant map onto the affine space $\mathbb{C}^{\dim(Z)}$ by composing $\psi$ with a $\dim(Z)$-dimensional generic linear projection.  The inclusion of a generic projection makes the new level sets consist of disjoint unions of level sets of the original $\psi$.  Since all results are local, the results apply to each level set of the disjoint union independently, and it is valid to assume that $\psi$ is dominant.

 Our main goal is to construct a Whitney stratification \cite{Whitney1965}
\[\mathcal S:Y_0 \subseteq Y_1\subseteq \cdots \subseteq Y_{m-1} \subseteq Y_m = Y\] such that $\Delta \subseteq Y_{m-1}$ where $m=\dim Y$.  Then a version of Thom's isotopy lemma, see \cite[Theorem 3.4]{ThomsIsotopyComplexAlgebra21}, applied to $\mathcal S$ and $\psi$  implies that $\psi$ is locally trivial outside of a closed and nowhere dense set $K(\psi,\cS)$, which may be taken to be a proper subvariety.  When $u\in\cU$ is generic, this implies that there exists a Euclidean ball $W$ containing $w\vcentcolon=\psi(u)$, and a homeomorphism $h:V_W\rightarrow V_w\times W$ which commutes with both $\psi$ and the projection onto the second factor of $V_w\times W$.  Moreover, the homeomorphism preserves the strata $S_i=Y_i- Y_{i-1}$.  Therefore, the deformation retraction of $W$ to $w$ extends to $V_w\times W$ and induces a deformation retraction of $V_W$ onto $V_w$ which preserves the strata.

The existence of this deformation retraction implies that the inclusion $\iota:V_w\hookrightarrow V_W$ induces an \emph{isomorphism}
\[
\iota_*\colon \Pi_1(V_w - Y_{m-1},u) \xrightarrow{\cong} \Pi_1(V_W- Y_{m-1},u).
\]
Moreover, since $Y_{m-1}$ is of real codimension 2 in $Y$, $V_w-Y_{m-1}$ and $V_w-\Delta$ differ in codimension at least $2$.  This implies that the inclusion $V_w-Y_{m-1}\hookrightarrow V_w-\Delta$ induces a surjection of fundamental groups.  A similar statement holds when working with $V_W$.  Putting this together, we have the following commutative diagram:
\begin{center}
\begin{tikzcd}
\Pi_1(V_w-Y_{m-1},u) \arrow[d,"\iota_\ast"',"\rotatebox{-90}{$\cong$}"] \arrow[r,two heads] & \Pi_1(V_w-\Delta,u) \arrow[d,"\iota_\ast"']\\
\Pi_1(V_W-Y_{m-1},u) \arrow[r,two heads] & \Pi_1(V_W-\Delta,u).
\end{tikzcd}
\end{center}
By extending by zeros to the right and applying the 4-lemma, we conclude that $\iota_\ast$ on the right is surjective.

We complete the proof by constructing the required Whitney stratification. We adapt the approach in \cite[Section 2]{ThomsIsotopyComplexAlgebra21}.  There, the authors construct a Whitney stratification by finding a (complex) codimension-$1$ subvariety $Y_{m-1}$ which contains the singular points of $Y$.  We change their approach by forcing $Y_{m-1}$ to also include $\Delta$ and the critical values of $\pi$.  We recall that both $\Delta$ and the critical values of $\pi$ are of codimension at least $1$, see, for example, Sard's theorem in \cite[Proposition 14.4]{Har92}.  Since the approach in \cite[Section 2]{ThomsIsotopyComplexAlgebra21} allows $Y_{m-1}$ to be any codimension-$1$ variety containing the singular points in $Y$, there is enough freedom to choose $Y_{m-1}$ as described here.  The remaining steps to construct $Y_i$ for $i<m-1$ of the Whitney stratification are unchanged from \cite[Section 2]{ThomsIsotopyComplexAlgebra21}.  We note that our change to the construction of $Y_{m-1}$ does not preserve the degree bounds of \cite{ThomsIsotopyComplexAlgebra21}, but this is not necessary in our application.

With this new $\cS$ in hand, the computation of $K=K(\psi,\cS)$ proceeds as in \cite{ThomsIsotopyComplexAlgebra21}.  By \cite[Theorem 4.1]{KurdykaOrroSimon2000}, this $K$ is contained in a codimension-1 subvariety of $Z$. 
\end{proof}

\begin{corollary}
\label{cor:same_monodromy_strip}
    Under the assumptions and notation of \Cref{lem:generic_strip_fundamental_equivalence}, the image of $\Pi_1(V_W-\Delta,u)$ under the monodromy homomorphism is $G_{\pi_{V_w},u}$. 
\end{corollary}
\begin{proof}
Let $\sigma:[\tau]\mapsto \sigma_\tau$ be the monodromy homomorphism of the monodromy action on the fibre $X_u$. By \Cref{lem:generic_strip_fundamental_equivalence}, every element of $\Pi_1(V_W-\Delta,u)$ has a representative in $\Pi_1(V_w-\Delta,u)$, which implies the following subgroup containment \[\sigma(\Pi_1(V_W-\Delta,u)) \leq \sigma(\Pi_1(V_w-\Delta,u))=G_{\pi_{V_w},u}.\] The other containment is immediate since  $V_W-\Delta$ contains $V_w-\Delta$. 
\end{proof}

From the technical topological result of \Cref{lem:generic_strip_fundamental_equivalence}, we prove a key result regarding block systems of restricted monodromy groups: that elements of the fundamental groupoid of $\mathcal U$ must induce maps on block systems over level sets. 

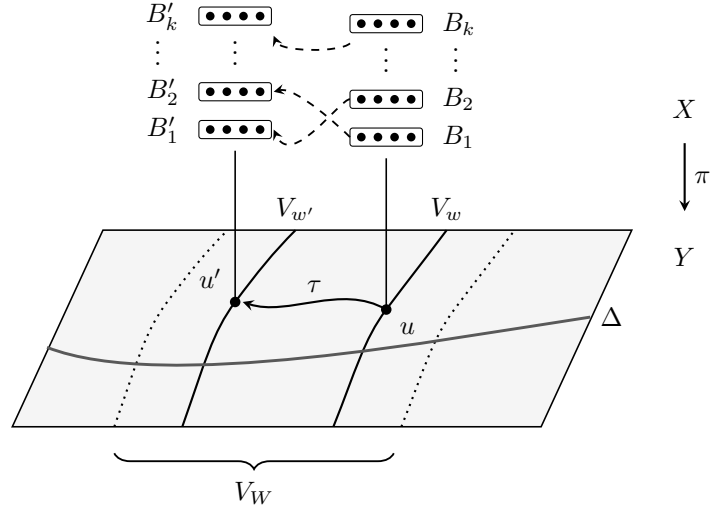
\begin{figure}[!htpb]
\begin{center}
\begin{tikzpicture}[
    scale=1,
    >=stealth,
    floor/.style={fill=gray!8, draw=black, line width=0.6pt},
    curve/.style={draw=black, line width=0.8pt},
    translatecurve/.style={draw=black, dotted, line width=0.8pt},
    deltacurve/.style={draw=gray!70!black, line width=1.1pt},
    point/.style={circle, fill=black, inner sep=1.4pt},
    lift/.style={draw=black, line width=0.6pt},
    box/.style={draw=black, rounded corners=1pt, inner sep=2pt},
    tau/.style={draw=black, line width=0.8pt, ->},
    fiberpath/.style={draw=black, dashed,line width=0.7pt, ->}
]
\coordinate (A) at (0,0);
\coordinate (B) at (7,0);
\coordinate (C) at (8.2,2.6);
\coordinate (D) at (1.2,2.6);
\draw[floor] (A) -- (B) -- (C) -- (D) -- cycle;
\node at (8.9,2.25) {$Y$};
\node at (8.9,4.2) {$X$};
\draw[translatecurve]
    (1.35,0)
    .. controls (1.65,0.85) and (1.75,1.25) .. (2.05,1.65)
    .. controls (2.35,2.05) and (2.55,2.30) .. (2.85,2.6);
\draw[curve]
    (2.25,0)
    .. controls (2.55,0.85) and (2.65,1.25) .. (2.95,1.65)
    .. controls (3.25,2.05) and (3.45,2.30) .. (3.75,2.6)
    node[above] {$V_{w'}$};
\draw[curve]
    (4.25,0)
    .. controls (4.55,0.75) and (4.65,1.15) .. (4.95,1.55)
    .. controls (5.25,1.95) and (5.45,2.20) .. (5.75,2.6)
    node[above] {$V_w$};
\draw[translatecurve]
    (5.15,0)
    .. controls (5.45,0.75) and (5.55,1.15) .. (5.85,1.55)
    .. controls (6.15,1.95) and (6.35,2.20) .. (6.65,2.6);
\draw[deltacurve]
    (0.46,1.05)
    .. controls (2.00,0.45) and (5.65,1.15) ..
    (7.65,1.45)
    node[right] {$\Delta$};
\coordinate (u)  at (4.95,1.55);
\coordinate (up) at (2.95,1.65);
\coordinate (upR) at (3.05,1.65);
\node[point, label=below right:{$u$}] at (u) {};
\node[point, label=above left:{$u'$}] at (up) {};
\draw[->, line width=0.8pt]
    (8.9,3.75) -- node[right] {$\pi$} (8.9,2.85);
\draw[tau]
    (u) .. controls (4.35,1.90) and (3.65,1.45) ..
    node[above] {$\tau$} (upR);
\coordinate (Utop)  at (4.95,3.55);
\coordinate (UPtop) at (2.95,3.65);
\draw[lift] (u) -- (Utop);
\draw[lift] (up) -- (UPtop);
\foreach \y/\lab in {0.00/{B_1},0.50/{B_2},1.5/{B_k}} {
    \draw[box]
        ($(Utop)+(-0.48,0.16+\y)$)
        rectangle
        ($(Utop)+(0.48,0.40+\y)$);
    \foreach \c in {0,1,2,3} {
        \node[point, inner sep=1.15pt]
            at ($(Utop)+(-0.33+0.22*\c,0.28+\y)$) {};
    }
    \node[right] at ($(Utop)+(0.62,0.28+\y)$) {$\lab$};
}
\node[right] at ($(Utop)+(0.75,1.4)$) {$\vdots$};
\node at ($(Utop)+(0,1.4)$) {$\vdots$};
\foreach \y/\lab in {0.00/{B'_1},0.50/{B'_2},1.5/{B'_k}} {
    \draw[box]
        ($(UPtop)+(-0.48,0.16+\y)$)
        rectangle
        ($(UPtop)+(0.48,0.40+\y)$);
    \foreach \c in {0,1,2,3} {
        \node[point, inner sep=1.15pt]
            at ($(UPtop)+(-0.33+0.22*\c,0.28+\y)$) {};
    }
    \node[left] at ($(UPtop)+(-0.62,0.28+\y)$) {$\lab$};
}
\node[left] at ($(UPtop)+(-0.85,1.4)$) {$\vdots$};
\node at ($(UPtop)+(0,1.4)$) {$\vdots$};
\draw[fiberpath]
    ($(Utop)+(-0.48,0.28)$)
    .. controls (4.35,3.95) and (3.75,4.5) ..
    ($(UPtop)+(0.51,0.78)$);
\draw[fiberpath]
    ($(Utop)+(-0.48,0.78)$)
    .. controls (4.35,4.35) and (3.7,3.40) ..
    ($(UPtop)+(0.51,0.28)$);
\draw[fiberpath]
    ($(Utop)+(-0.48,1.52)$)
    .. controls (4.35,4.95) and (3.55,4.95) ..
    ($(UPtop)+(0.51,1.52)$);
\draw[decorate, decoration={brace, mirror, amplitude=6pt}, line width=0.7pt]
    (1.35,-0.35) -- node[below=7pt] {$V_W$} (5.05,-0.35);
\end{tikzpicture}
\end{center}
\caption{An illustration of the claim
of \Cref{lem:blocks_to_blocks_in_leaves}. For paths within a neighborhood $V_W$ of $V_w$, blocks of $G_{\pi_{V_w}}$ map to blocks of $G_{\pi_{V_{w'}}}$ under the bijection $\sigma_{\tau}$ corresponding to conjugation by a path $\tau$ from $u$ to $u'$ in $\mathcal U$. }
    \label{fig:blocks_to_blocks_in_leaves}
\end{figure}

\begin{lemma}
    \label{lem:blocks_to_blocks_in_leaves}
    Suppose $\pi:X\rightarrow Y$ is a branched cover, $\psi: Y \to Z$ is a $\pi$-generic dominant map, and $u \in \mathcal U$ is generic. Suppose that $G_{\pi_u,u}$ contains a transposition and is not $S_{X_u}$.  For $w = \psi(u)$, let $W \subseteq Z$ be a Euclidean ball around $w$ as in \Cref{lem:generic_strip_fundamental_equivalence}.  Suppose that $u,u' \in V_W-\Delta$, $[\tau]\in\Pi_1(\cU\cap V_W,u,u')$, and $w'=\psi(u')$. Then the bijection $\sigma_\tau$ maps the unique minimal blocks of $G_{\pi_{V_w}}$ to the unique minimal blocks of $G_{\pi_{V_{w'}}}$.  
\end{lemma}
\begin{proof}
    In this proof, for ease of notation, we write $\pi_u$ for $\pi_{V_{w}}$ and $\pi_{u'}$ for $\pi_{V_{w'}}$.  Since $V_{w}$ and $V_{w'}$ are homeomorphic, $V_{w}\times W$ and $V_{w'}\times W$ are also homeomorphic.  In addition, these homeomorphisms preserve the Whitney stratification of \Cref{lem:generic_strip_fundamental_equivalence}.  Therefore, using the same $W$, the conclusions of \Cref{lem:generic_strip_fundamental_equivalence} also apply for $u'$, that is, the inclusion of spaces induces the surjection $\Pi_1(V_{w'}-\Delta,u')\twoheadrightarrow\Pi_1(V_W-\Delta,u')$.  
    
    Since $V_w$ is irreducible, it follows that $V_w\times W$ is connected, and hence, $V_W$ is also connected.  Since $\Delta$ is of real codimension at least $2$ in $Y$, $V_W-\Delta$ is also connected.  Therefore, $\Pi_1(V_W-\Delta,u)$ and $\Pi_1(V_W-\Delta,u')$ are conjugate via the given path $\tau$.  Then, by \Cref{cor:same_monodromy_strip}, it follows that $G_{\pi_u,u}$ and $G_{\pi_{u'},u'}$ are also conjugate.  Therefore, since $G_{\pi_u,u}$ is transitive and contains a transposition, $G_{\pi_{u'},u'}$ is also transitive and contains a transposition.  Hence, by \Cref{cor:uniformorimprimitive}, both $G_{\pi_u,u}$ and $G_{\pi_{u'},u'}$ are either full symmetric groups or they have unique minimal block systems of the same size.  Since we assumed that $G_{\pi_u,u}$ is not $S_{X_u}$, both $G_{\pi_u,u}$ and $G_{\pi_{u'},u'}$ are imprimitive.

    Let $B$ be a block of the block system $\cB_u$ of $G_{\pi_u,u}$, see \Cref{fig:blocks_to_blocks_in_leaves} for a visualization. Suppose towards a contradiction, that $\sigma_{\tau}(B)$ is not a part of the block system $\mathcal B_{u'}$ of $G_{\pi_{u'},u'}$. Without loss of generality, assume that $x \in \sigma_{\tau}(B)$ belongs to the block $B_1'$ of $\mathcal B_{u'}$ and $x'\in\sigma_\tau(B)$ belongs to a distinct block.     
    Since $G_{\pi_{u'},u'}$ is a full wreath product and the blocks of $\cB_u$ and $\cB_{u'}$ are the same size, there exists a monodromy element $\sigma_{\gamma} \in G_{\pi_{u'},u'}$ which stabilizes $x'$ within its block but moves $x$ to an element $\hat{x}$ not belonging to $\sigma_{\tau}(B)$. Consequently, the monodromy permutation $
\sigma_{\tau^{-1} \sqcdot \gamma \sqcdot \tau} \in G_{\pi}$ 
    sends $B$ to a non-block of $\cB_u$.  By \Cref{lem:generic_strip_fundamental_equivalence}, since $[\tau^{-1} \sqcdot \gamma \sqcdot \tau] \in \Pi_1(V_W-\Delta,u)$,  there exists a representative $\gamma'$ of this class inside $\Pi_1(V_{w}-\Delta,u)$, a contradiction since  $\sigma_{\gamma'}(B)=\sigma_{\tau^{-1} \sqcdot \gamma \sqcdot \tau}(B)$ is not a block of $G_{\pi_u,u}$.
\end{proof}

The following theorem proves the fourth part of \Cref{thm:main_theorem}.

\begin{theorem}\label{thm:preserves_blocks}
    Let $\pi:X\rightarrow Y$ be a branched cover of degree $d$, $\psi: Y \to Z$ a $\pi$-generic dominant map, $u\in\cU$ a generic point, and $V$ the level set of $\psi$ containing $u$.  If $G_{\pi_V,u}$ contains a transposition and is not $S_{X_u}$, then every element of $\Pi_1(\cU,u)$ preserves the unique minimal block system of $G_{\pi_V}$. 
\end{theorem}
\begin{proof}
For ease of notation, we write $\pi_u$ for $\pi_{V_{\psi(u)}}$.  Let $[\tau] \in \Pi_1(\mathcal U,u)$ and $\cB_u$ be the minimal block system of $G_{\pi_u,u}$.  By replacing $\tau$ with a slight perturbation, if necessary, to avoid $K(\psi,\mathcal{S})$ from \Cref{lem:generic_strip_fundamental_equivalence}, if necessary, let $\gamma=\psi\circ\tau$.
One may cover $\gamma$ by the open sets of \Cref{lem:generic_strip_fundamental_equivalence} which, by the compactness of $\gamma$, admit a finite subcover.  We next find a partition of $\gamma$ subordinate to this subcover and apply \Cref{lem:blocks_to_blocks_in_leaves} to each element of this partition.  We conclude that for every $t$, $\tau$ takes the block system of $G_{\pi_u,u}$ to a block system of $G_{\pi_{\tau(t)},\tau(t)}$ of the same size.  Hence, these blocks are preserved over the entire loop and $\{\sigma_\tau(B)\}_{B\in\cB_u}$ is a block system of $G_{\pi_u,u}$. By the uniqueness of the minimal block system on $G_{\pi_u,u}$, these two block systems of the same size must be equal, that is, $\{\sigma_\tau(B)\}_{B\in\cB_u}=\cB_u$, and hence the block system is preserved by $\sigma_\tau$. 
\end{proof}

Putting these results together proves the final statement in \Cref{thm:main_theorem} and establishes the full symmetric group or imprimitive dichotomy in the presence of transpositions.

\begin{theorem}
\label{thm:iff-full-symmetric}
Let $\pi:X\rightarrow Y$ be a branched cover, $\psi:Y\rightarrow Z$ a $\pi$-generic dominant map, and $w\in Z$ be generic.  If $G_{\pi_{V_w}}$ contains a transposition, then $G_\pi$ is isomorphic to the full symmetric group if and only if $G_{\pi_{V_w}}$ is isomorphic to the full symmetric group.
\end{theorem}
\begin{proof}
For simplicity of notation, let $V\vcentcolon=V_w$.  Since $G_{\pi_V}\subseteq G_\pi$, the backward direction is immediate.  For the forward direction, by \Cref{cor:uniformorimprimitive} $G_{\pi_V}$ is isomorphic to the full wreath product $S_k\wr H$ with $k>1$ for some transitive $H\leq S_b$.  Now, we assume that $G_{\pi_V}$ is not $S_d$, which implies that $1<k<d$.  By \Cref{cor:uniformorimprimitive}, we conclude that $G_{\pi_V}$ is imprimitive with minimal block size $k$.  By \Cref{thm:preserves_blocks}, for any $[\tau]\in\Pi_1(\cU,u)$, $\sigma_\tau$ preserves the block system $\cB_u$.  Hence, $G_\pi\leq \textrm{Wr}(\cB_u)\cong S_k\wr S_b$ where $b=d/k$.  Hence $G_\pi$ is not the full symmetric group.
\end{proof}

\begin{remark}
In the case where the assumptions of \Cref{thm:iff-full-symmetric} hold and the restricted monodromy group for a generic level set is $S_k\wr H$ for some transitive $H\leq S_b$, \Cref{thm:iff-full-symmetric} does not guarantee that the restricted monodromy groups and $G_\pi$ are the same.  In fact, we may only conclude that $S_k\wr H\leq G_\pi\leq S_k\wr S_b$.
\end{remark}

\begin{remark}
A common way in the literature to argue that a monodromy group is the full symmetric group is to show that the monodromy group is 2-transitive and contains a transposition, see \cite[p.111]{Arb85} and \cite{Harris79}.  Our approach from \Cref{thm:iff-full-symmetric} provides a new way to confirm that a monodromy group is the full symmetric group.
\end{remark}

\subsection{Transpositions and ramification}
One of the key conditions in \Cref{thm:iff-full-symmetric} is that the restricted monodromy group contains a transposition. Locally, a transposition occurs when a fibre of $\pi$ over $y \in \Delta$ consists of $d-2$ simple points and one of multiplicity two which is smooth on $X$. In this case, a small loop around $y$ gives a transposition exchanging the two sheets which intersect over $y$. We call a fibre of a degree-$d$ branched cover $\pi\colon X\to Y$ \mydefit{simply ramified} if the fibre $\pi^{-1}(y)$ consists of $d-1$ isolated points, one of which has multiplicity 2.  We note that in this definition, $X$ is not required to be smooth at the point of multiplicity.

\begin{lemma}
    \label{lem:smooth-simplyramified-transp}
    Let $\pi\colon X\to Y$ be a branched cover with branch locus $\Delta$. Let $y\in \Delta$ be such that $\pi$ is simply ramified over $y$, witnessed by $x\in \pi^{-1}(y)$ such that $X$ is smooth at $x$. Then in any open neighborhood of $y$, there exists a loop $\gamma$ which induces a transposition in $G_\pi$.
\end{lemma}

\begin{proof}
    In \cite[Section II.3, p.698]{Harris79}, Harris proves the existence of a transposition in the monodromy group when $X$ is locally irreducible at the double point of a simply ramified fibre via a small loop around $y$, as in the statement.  Since $X$ is smooth at $x$, $X$ cannot be locally reducible at $x$, so we conclude that $G_\pi$ has a transposition.
\end{proof}

\begin{remark}
We note that a point $x \in \pi^{-1}(y)$ being singular in its fibre does not contradict $x$ being smooth on the incidence variety $X$. Take, for example, the point $x=0$ in the fibre of $\mathcal V(y-x^2) \rightarrow \mathbb{C}_y$ over $y=0$.
\end{remark}

We proceed to  adapt \Cref{lem:smooth-simplyramified-transp} to the restricted case by first showing that smoothness of $x\in X$ in \Cref{lem:smooth-simplyramified-transp} translates to smoothness of $X_V$ when $V$ is a generic level set.  The following is a Bertini-style statement that gives conditions connecting the smoothness on $X$ to that of $X_V$.  We provide a full proof for completeness.

\begin{proposition}
    \label{prop:restriction-smoothness}
    Let $\pi\colon X\to Y$ be a branched cover with branch locus $\Delta\subseteq Y$ and $V\subseteq Y$ be a generic level set of a $\pi$-generic map. Let $y\in V\cap \Delta$ be a generic point of an irreducible component $\Delta_\mu$ of $\Delta$ which is smooth on both $V$ and $Y$. Furthermore, suppose that $\pi$ is simply ramified over $y$, with $x\in\pi^{-1}(y)$ having multiplicity two.  If $X$ is smooth at $x$ and $\Delta_\mu$ is of (complex) codimension $1$ in $Y$, then $X_V$ is also smooth at $x$.
\end{proposition}

\begin{proof}
Since $\Delta_\mu$ intersects a generic level set of $\psi$, $\psi|_{\Delta_\mu}$ is a dominant map.  Therefore, by Sard's theorem and a dimension argument, for generic $z\in Z$, $\Delta_\mu$ intersects $\psi^{-1}(z)$ transversely.  Throughout the remainder of this proof, we assume that $X$ and $Y$ are embedded in the projective spaces $\mathbb{P}^m$ and $\mathbb{P}^n$, respectively, and we restrict our attention to affine charts in $\mathbb{C}^m$ and $\mathbb{C}^n$.

We first replace $X$ with an isomorphic copy $X'\subseteq\mathbb{C}^m\times\mathbb{C}^n$, where $X'$ consists of the pairs of points $(x,y)$ with $x\in X$ and $y=\pi(x)\in Y$.  In this case, $\pi':X'\rightarrow Y$ is the projection onto the last $n$ entries of a point of $X'$.  We suppose that $X'$ is given by the equations $\mathcal{H}\vcentcolon=\{h_1,\dots,h_{\ell}\}$, where $\ell\geq m+n-\dim X$.  Since $x$ is a smooth point of $X$, the Jacobian $J_{X'}(x,y)=\begin{pmatrix}\partial_X\mathcal{H}(x,y)&\partial_Y\mathcal{H}(x,y)\end{pmatrix}$ is of rank $m+n-\dim X$.  Here, we use $\partial_X$ to indicate the derivative with respect to the first $m$ variables, that is, those variables associated to $X$, and $\partial_Y$ to indicate the derivative with respect to the last $n$ variables, that is, those variables associated to $Y$.  Since $x$ is a double point of the fibre and fibres are zero-dimensional, the first block $\partial_X\mathcal{H}(x,y)$ of the Jacobian has rank $m-1$.  

Using these ranks, we conclude that the \emph{row span} of the Jacobian intersected with $0^m\times\mathbb{C}^n$ is an $(n-\dim X+1)$-dimensional space, which we define as $0^m\times N$.  For any vector $(0,v)$ in $0^m\times N$, every vector in the tangent space $T_{X'}(x,y)$ must be perpendicular to $(0,v)$.  Since the pushforward of a projection map is still a projection map, this implies that all vectors in $\pi'_\ast(T_{X'}(x,y))$ must be perpendicular to $v$. The space of vectors in $\mathbb{C}^n$ perpendicular to $N$ is a $(\dim X-1)$-dimensional space $T$, which we next show is the tangent space of $\Delta_\mu$ at $y$.

From an alternate perspective, suppose that we construct the system of equations consisting of $\mathcal{H}$ along with those polynomials corresponding to the rank condition on $\partial_X\mathcal{H}$.  This system corresponds to a subvariety $\Lambda$ of $X'$ such that $\pi(\Lambda)\subseteq\Delta$.  Since a generic point of $\Delta_\mu$ is simply ramified and the witness to this ramification is in $\Lambda$, an open subset of $\Delta_\mu$ is contained in $\pi(\Lambda)$.  In fact, there must be an irreducible component $\Lambda_\mu$ of $\Lambda$ that maps dominantly onto $\Delta_\mu$.  Since $y$ is generic in $\Delta_\mu$, it follows that $(x,y)\in\Lambda_\mu$.  Since $y$ is generic, $\Delta_\mu$ is smooth at $y$ and, by Sard's theorem, $\left(\pi'|_{\Lambda_\mu}\right)_\ast$ maps $T_{\Lambda_\mu}(x,y)$ onto $T_{\Delta_\mu}(y)$.  By the argument above, we know that every vector in $\left(\pi'|_{\Lambda_\mu}\right)_\ast\left(T_{\Lambda_\mu}(x,y)\right)$ is perpendicular to the vectors in $N$.  In other words, the tangent space of $\Delta_\mu$ at $y$ is contained in $T$ and is at most $(\dim X-1)$-dimensional.  Since $\Delta_\mu$ is codimension 1 in $Y$, these dimensions match, and we conclude that the tangent space of $\Delta_\mu$ is exactly $T$.

Finally, suppose that $V$ is given by the equations $\mathcal{F}=\{f_1,\dots,f_k\}$.  Since $y$ is a smooth point of $V$, the rank of the Jacobian $J_V(y)=\partial_Y\mathcal{F}(y)$ is $n-\dim V$.  Since $V$ intersects $\Delta_\mu$ transversely at $y$ and $\Delta_\mu$ is of codimension $1$ in $Y$, there is some vector $w\in T_V(y)$ which is not in $T_{\Delta_\mu}(y)$.  We now observe that $X'_V$ is defined by the system of equations $\mathcal{H}\cup \mathcal{F}$, and its Jacobian at $(x,y)$ is
$$
J_{X'_V}(x,y)=\begin{pmatrix}
\partial_X\mathcal{H}(x,y)&\partial_Y\mathcal{H}(x,y)\\
0&\partial_Y\mathcal{F}(y)
\end{pmatrix}.
$$
By the previous dimension calculations, we see that the rank of this Jacobian is $m-1$ plus the dimension of the intersection of the row span with $0^m\times\C^n$.  Since $\partial_Y\mathcal{F}(y)$ has rank $n-\dim V$, the rank of this Jacobian is at least $n+m-\dim V-1$.  With a little more care, we see that the rank of this Jacobian is $m-1$ plus the dimension of $(N$ plus the row span of $\partial_Y\mathcal{F}(y))$.  In other words, the rank of the Jacobian is
\begin{multline*}
(m-1)+\dim N+\operatorname{rank}\partial_Y\mathcal{F}(y)-\dim\left(N\cap\operatorname{row}\operatorname{span}\partial_Y\mathcal{F}(y)\right)\\
=m+2n-\dim X-\dim V-\dim\left(N\cap\operatorname{row}\operatorname{span}\partial_Y\mathcal{F}(y)\right).
\end{multline*}
In the remaining dimension calculation, since both $N$ and $\operatorname{row}\operatorname{span}\partial_Y\mathcal{F}(y)$ correspond to normal spaces of varieties within $Y$, they must both include the $(n-\dim Y)$-dimensional normal space to $Y$.  Therefore,
$$
n-\dim Y\leq \dim\left(N\cap\operatorname{row}\operatorname{span}\partial_Y\mathcal{F}(y)\right)\leq \dim N=n-\dim Y+1.
$$
Hence, the dimension of the intersection of interest is either $n-\dim Y$, when $N$ is not contained in $\operatorname{row}\operatorname{span}\partial_Y\mathcal{F}(y)$, or $n-\dim Y+1$, when this containment holds.  If the containment were to hold, this would contradict the existence of the vector $w$ since then $w$ would need to be perpendicular to $N$, that is, $w$ would be in $T_{\Delta_\mu}(y)$, which is impossible.  Therefore, the dimension of the intersection is $n-\dim Y$, and the rank of the Jacobian is $m+n-\dim V$, that is, $X'_V$ is smooth at $(x,y)$.
\end{proof}

\begin{remark}
We note that since components of $\Delta$ of codimension greater than one do not contribute to the fundamental group of $Y-\Delta$, it suffices to consider the components of $\Delta$ of codimension 1, as in the statement of \Cref{prop:restriction-smoothness}.
\end{remark}

\begin{corollary}
    \label{cor:restricted-mon-gp-transposition}
Under the hypotheses of \Cref{prop:restriction-smoothness}, $G_{\pi_V}$ contains a transposition.
\end{corollary}

\begin{proof}
    Since $\psi$ is $\pi$-generic, $\pi_V:X_V\rightarrow V$ is a branched cover and the fibers $\pi^{-1}(y)$ and $\pi_V^{-1}(y)$ are equal.  Since $\pi$ is simply ramified over $y$ and $\Delta_\mu$ intersects $V$ transversely at $y$, $\pi_V$ is also simply ramified over $y$.  Since $X$ is smooth at $x$, \Cref{prop:restriction-smoothness} implies that $X_V$ is smooth at $x$.  Then, applying \Cref{lem:smooth-simplyramified-transp} to $\pi_V$ at $y$ shows that $G_{\pi_V}$ contains a transposition.
\end{proof}

\begin{example}
    Having a local transposition as a consequence of \Cref{cor:restricted-mon-gp-transposition} is not the only way a monodromy group can contain a transposition. For example, consider the branched cover $\pi:X \to \mathbb{C}$ from the plane curve 
    \[
    X = \mathcal V(10y^6-24y^5+15y^4-x) = \mathcal V(f(x,y))
    \]
    to $\mathbb{C}$ defined by the coordinate projection $\pi(x,y) = x$. This map has a branch locus given by 
    \[
    \mathcal V\left(\textrm{Res}\left(f,\frac{\partial f}{\partial y}; y\right)\right) = \mathcal V(x^3(x-1)^2) =\{0,1\}
    \]
    The local monodromy over $x=0$ has cycle type $(4,1,1)$ and the local monodromy over $x=1$ has cycle type $(3,1,1,1)$. Specifically, the monodromy elements are $\sigma_0 = (3,4,5,6)$ and $\sigma_1 = (1,2,3),$ which generate $S_6$.
\end{example}

\begin{remark}
An alternate way in which the monodromy group of a generic level set $V$ of $\psi$ might contain a transposition is if $V$ is not simply connected. An example of this is the realification of the monodromy of the square root function restricted to the unit circle:
\[
\mathcal V(y_1+\sqrt{-1}y_2-x^2,y_1^2+y_2^2-1) \mapsto \mathbb{S}^1 = \mathcal V(y_1^2+y_2^2-1)
\]
This branched cover has no ramification, but its monodromy group is the full symmetric group.
\end{remark}
 
\subsection{Generalized Trace Tests for Branched Covers}\label{sec:generalizedtrace}
The classical \textit{trace test} is an algorithm in numerical algebraic geometry that determines whether a subset of  a generic zero-dimensional linear slice of an irreducible affine variety $\mathcal M \subseteq \mathbb{C}^n$ is the entire intersection. The test evaluates the \mydefit{trace function}, given by coordinate-wise summation, on a nonempty subset $P$ of a linear intersection $\mathcal M \cap L$,
\begin{align*}
\mydef{\Sigma}: \textrm{PowerSet}_{\textrm{finite}}(\mathbb{C}^n) &\to \mathbb{C}^n \\
P &\mapsto \sum_{p \in P} p
\end{align*}
The test then analytically continues $P$ over a generic path $L_t$ of parallel linear spaces. That is, setting $P_t \subseteq \mathcal M \cap L_t$ to be the set of those continuations, the trace test evaluates $\Sigma(P_t)$ at several $t$-values. The \mydefit{classical trace test} then asserts that $P = \mathcal M \cap L$ if and only if $\Sigma(P_t)$ is a linear function of $t$.

More generally, a \mydefit{trace test} is an algorithm that can test whether a subset of a fibre of a branched cover is the entire fibre. In \cite{STT23}, the classical trace test was extended to the setting of sparse polynomial systems. In that work, the analogue of a pencil of parallel linear spaces is a family of sparse systems where some coefficients are allowed to vary. The authors of \cite{STT23} give sufficient conditions on which sets coefficients to vary induce a sparse trace test. In \Cref{sec:sparse_trace_test}, we combine the general theory of trace tests of branched covers, as developed in this section, with our main result to show that \textit{any} single coefficient suffices to induce a sparse trace test. 

Let $X \subseteq \mathbb{C}^n$ be an affine variety and $\pi:X \to Y$ a branched cover of degree $d$. Let $y^\ast\in Y$ be generic and consider a subset $P \subseteq X_{y^\ast}$ of the fibre of $\pi$ over $y^\ast$.  We recall that for any $[\tau]$ in the universal cover $\Pi_1(\mathcal U,y^\ast,\cdot)$ based at $y^*$, the bijection $\sigma_\tau$ is well-defined.  Thus, it is natural to transport $P$ via a path $\tau$ and compute the trace of the endpoints of those paths. This results in the function $\Xi_{P,\pi}$ (where the $\pi$ is dropped when it is clear from context) defined by
\begin{align}
    \mydef{\Xi_{P}}: \Pi_1(\mathcal U,y^\ast,\cdot) &\to \mathbb{C}^n \\ 
    [\tau] &\mapsto \Sigma(\sigma_{\tau}(P)). \nonumber
\end{align}
We say that $\Xi_{P}$ \mydefit{descends to a function} on $\mathcal U$ if $\Xi_P([\tau]) = \Xi_P([\tau'])$ whenever $\tau(1)=\tau'(1)$. In other words, the map $\Xi_P$ factors through the map $\mydef{\textrm{ev}_1}:\Pi_1(\mathcal U,y^{\ast},\cdot) \to \mathcal U$ defined by $\textrm{ev}_1([\tau]) = \tau(1)$ such that $\Xi_P = \xi_P \circ \textrm{ev}_1$ for some $\mydef{\xi_P}: \mathcal U \to \mathbb{C}^n$. This is summarized in the following commutative diagram 
\begin{center}
\begin{tikzpicture}[
    >=stealth,
    node distance=2.8cm and 4.8cm
]
\node (A) {$\Pi_1(\mathcal U,y^\ast,\cdot)$};
\node (B) [right=of A] {$\mathbb C^n$};
\node (C) [below=of A] {$\mathcal U$};
\draw[->] (A) -- node[above] {$\Xi_P$} (B);
\draw[->] (A) -- node[left] {$\operatorname{ev}_1$} (C);
\draw[->,dotted] (C) -- node[below right] {$\xi_P$} (B);
\node (a) [below right=0.1cm and -0.5cm of A] {$[\tau]$};
\node (b) [below left =0.1cm and 1.2cm of B] {$\Sigma\left(\sigma_\tau(P)\right)$};
\node (c) [above right =0.3cm and 0.1cm of C] {$\tau(1)$};
\draw[|->] (a) -- (b);
\draw[|->] (a) -- (c);
\draw[|->,dotted] (c) -- (b);
\end{tikzpicture}
\end{center}
If $\Xi_P$ descends to a function $\xi_P$ on $\mathcal U$, this implies that the coordinates functions of  $\xi_P$ belong to the \textit{function field} $\mathbb{C}(Y)$. In many practical settings $Y = \mathbb{C}^k$, so $\xi_P$ is an $n$-tuple of rational functions in $k$ variables. 

We now succinctly state the classical trace test. Let $\mathcal M \subseteq \mathbb{C}^n$ be an irreducible variety of codimension $k$, and let $\textrm{Gr}(k+1,n+1)$ be the Grassmannian of $(k+1)$-planes in $\mathbb{C}^{n+1}$. We interpret a generic element of $\textrm{Gr}(k+1,n+1)$ as an affine $k$-plane in $\mathbb{C}^n$ used to slice $\mathcal M$ and consider the branched cover 
\[
\pi: \mydef{W(\mathcal M)} \vcentcolon= \{(L,p) \in \textrm{Gr}(k+1,n+1) \times \mathcal M \mid p \in L\} \longrightarrow \textrm{Gr}(k+1,n+1)
\]
whose fibre over a generic linear space $L$ is identified with the $\textrm{deg}(\mathcal M)$-many  points of $\mathcal M \cap L$. We call $\pi$ the \mydefit{witness cover} associated to $\mathcal M$, since the datum of a fibre represents a \textit{witness set} in numerical algebraic geometry \cite{NAG:SommeseWampler}. Let $\psi: \textrm{Gr}(k+1,n+1) \dashrightarrow \textrm{Gr}(k,n)$ be the map which shifts an affine linear space to the origin. We observe that $\psi$ is $\pi$-generic and that a generic level set $V$ of $\psi$ is a family of parallel $k$-planes.  When restricting $\Xi_{P,\pi}$ to $\Xi_{P,\pi_V}$, we only consider paths in $\mathcal{U}\cap V$, that is, 
$$
\Xi_{P,\pi_V}:\Pi_1(\mathcal{U}\cap V,y^\ast,\cdot)\rightarrow\bC^n.
$$
We now give the statement of the classical trace test in our terminology. 

\begin{proposition}[The (classical) trace test]
\label{prop:classicaltracetest}
    Let $\mathcal M$ be an irreducible variety, $\pi$ its witness cover, $V$ a generic level set of $\psi$, as defined above, and $P$  a nonempty subset of a generic fibre $W(\mathcal M)_{L}$ of $\pi_V$. Then the following are equivalent
     \begin{itemize}
         \item $P = W(\mathcal M)_L$ 
         \item  $\Xi_{P, \pi_V}$ descends to a function on $\mathcal U \cap V$
         \item The closure of the image of $\Xi_{P,\pi_V}$ is an affine-linear subspace of $\mathbb C^n$ of dimension $\dim(V)$.
     \end{itemize}
\end{proposition}
In practice, one further restricts to a generic \textit{pencil} of $k$-planes inside such a level set, parametrized by $t \mapsto V_t$. In that setting, the conclusion of \Cref{prop:classicaltracetest} is often phrased as
$$
P_t = \mathcal M \cap V_t\textit{ if and only if }
\Sigma(P_t)\textit{ is an affine-linear function of }t.
$$

We now  generalize \Cref{prop:classicaltracetest} for a general branched cover of the form $\pi :X \to Y$ for $X \subseteq \mathbb{C}^n$. For a subset $P\subseteq X_{y^\ast}$, consider the traces of the
elements of its orbit under the $G_{\pi,y^\ast}$-action:
\[
    \Sigma(G_{\pi,y^\ast} P)
    =
    \{\Sigma(g(P)) \mid g\in G_{\pi,y^\ast}\}.
\]
We say that $P$ is \mydefit{trace-invariant} if
\begin{equation}
    \label{eq:traceinvariance}
    \Sigma(g(P))=\Sigma(P)
    \qquad
    \text{ for all } g\in G_{\pi,y^\ast}.
\end{equation}
We say that $\pi$ admits a \mydefit{(proper) trace-invariant subset}
if, for generic $y^\ast\in\mathcal U$, there exists a nonempty proper
subset $P\varsubsetneq X_{y^\ast}$ which is trace-invariant.

\begin{proposition}[Generalized Trace Test]
\label{prop:generalizedtracetest}
    Let $\pi:X\to Y$ be a branched cover of affine varieties,
    let $y^\ast\in\mathcal U$ be generic, and let
    $P\subseteq X_{y^\ast}$. Then
    \[
        \Xi_P \text{ descends to }\mathcal U
        \iff
        P \text{ is trace-invariant}.
    \]
    In particular, if $\pi$ admits no proper trace-invariant subsets, then for $
    \emptyset \neq P \subseteq X_{y^\ast}$
    \[ \Xi_P \text{ descends to }\mathcal U \iff P = X_{y^\ast}.
    \]
\end{proposition}

\begin{proof}
    Suppose first that $\Xi_P$ descends to $\mathcal U$. For every loop
    $\gamma$ based at $y^\ast$, $\gamma$ and the constant path $\iota$ at $y^\ast$ have the same endpoint.  Hence
    \[
        \Sigma\left(\sigma_\gamma(P)\right)
        =
        \Xi_P([\gamma])
        =
        \xi_P(y^\ast)
        =
        \Xi_P([\iota])
        =
        \Sigma(P).
    \]
    Since every $g\in G_{\pi,y^\ast}$ is induced by such a loop, it
    follows that
    \[
        \Sigma(g(P))=\Sigma(P)
        \qquad
        \forall g\in G_{\pi,y^\ast},
    \]
    so $P$ is trace-invariant.

Conversely, suppose that $P$ is trace-invariant. Then
$\Sigma\left(\sigma_\gamma(P)\right)=\Sigma(P)
$
for every loop $\gamma$ based at $y^\ast$.  
Since $y^\ast$ is generic,
 the equality $\Sigma(\sigma_{\gamma}(P))=\Sigma(P)$ of traces cannot be an isolated coincidence, and we provide a proof that this is an algebraic condition in \Cref{lem:coincidencealgebraicsubset}.  In particular, for any subset $\emptyset\not=Q\subsetneq X_{y^\ast}$, the set of $y\in\cU$ so that some translate $\sigma_\tau(Q)$ is trace-invariant forms an algebraic set.  Since there are finitely many subsets of $X_{y^\ast}$, we conclude that for $P$ to be trace-invariant at a generic point of $\cU$, every point in $\cU$ must have a translate of $P$ which is trace-invariant.

Now, suppose that $[\tau_1],[\tau_2]\in\Pi_1(\cU,y^\ast,y')$ such that $\sigma_{\tau_1}(P)$ is trace-invariant.  Then 
$$
\Sigma(\sigma_{\tau_2}(P))=\Sigma(\sigma_{\tau_2}(\sigma_{(\tau_1)^{-1}}(\sigma_{\tau_1}(P))))=\Sigma(\sigma_{(\tau_1)^{-1}\sqcdot\tau_2}(\sigma_{\tau_1}(P))).
$$
Since $(\tau_1)^{-1}\sqcdot\tau_2$ is a loop at $y'$ and $\sigma_{\tau_1}(P)$ is trace-invariant, this simplifies to $\Sigma(\sigma_{\tau_1}(P))$.  This, in turn, implies that $\Xi_P$ is independent of path and descends to a function on $\cU$.

Finally, if $\pi$ admits no proper nonempty trace-invariant subsets,
then only trivial trace-invariant subsets can make $\Xi_P$ descend.
Since $P$ is nonempty, it must be the entire fibre.
\end{proof}

The monodromy group $G_{\pi}$ gives some information about whether $\pi$ admits trace-invariant subsets. For example, if $G_{\pi}$ is the full symmetric group, then one can always swap one element in a proper non-empty $P$ with 
something outside of $P$ while stabilizing the remaining $|P|-1$ points set-wise. Such an action changes the trace and so $P$ cannot be trace-invariant, cf. \cite[Proposition~15]{Ley18}. 

Based on the machinery above, the classical trace test equivalence \[P = W(\mathcal M)_L \iff \Xi_{P,\pi_V} \text{ descends to a function on }\cU\cap V\] becomes easy to prove.  It is a consequence of the fact that the monodromy of a witness cover is the full symmetric group along with the fact that restricting a branched cover to a generic line in the base space retains the monodromy group. The remaining equivalence to linearity serves only as a method for making the algorithm effective since it is simple to check if a function is linear. In the generalization which follows, no assertion about the degree of the trace is required. However, $\Xi$ descending implies that the trace is a function of the base of the branched cover, and so bounding the degree of the relevant rational functions and interpolating would make a generalized trace test effective, see \Cref{sec:completionsparsetrace} for this approach in the sparse case. 

The property of being able to swap out singletons from a proper nonempty subset is not unique to the symmetric group. The alternating group, for example, has this property as well. Thus, we generalize and say that a permutation group $G \leq S_d$ is \mydefit{one-swappable} if for every proper nonempty subset $P$  of $[d]$, we have that there exists $g \in G$ such that $|P \cap g(P)|=|P|-1$. See \Cref{rem:one-swappable-groups} for a list of one-swappable groups which are not the symmetric or alternating groups. In addition to one-swappability, being $2$-transitive is sufficient to forbid proper trace-invariant subsets. 
\begin{lemma}
    \label{lem:sufficient_for_no_TIS}
    \hspace{-3pt}If $G_{\pi}$ \hspace{-5pt} is one-swappable or  $2$-transitive then $\pi$ admits no proper trace-invariant~subset.
\end{lemma}
\begin{proof}
    Let $y^\ast\in\mathcal U$ be generic and suppose, toward a contradiction, that
    $\emptyset\neq P\varsubsetneq X_{y^\ast}=\{x_1,\dots,x_d\}$ is trace-invariant.  Let $|P|=r$.  The argument for the full symmetric group above extends directly to one-swappable groups, proving the statement for one-swappable groups.
    
    Now suppose that $G_\pi$ is $2$-transitive. Let $\mathcal O$ denote the monodromy orbit of $P$.  Since $G_\pi$ is transitive, every $x_i$ is in the same number of elements of $\mathcal{O}$, which we denote by $\lambda_1$.  Similarly, since $G_\pi$ is $2$-transitive, every pair $(x_i,x_j)$ with $i\not=j$ appears in the same number of elements of $\mathcal{O}$, which we denote by $\lambda_2$.  Adding the occurrences over all elements or pairs, we arrive at the following:
    \[
    d\lambda_1=r|\mathcal{O}|\quad\text{and}\quad\frac{d(d-1)}{2}\lambda_2=\frac{r(r-1)}{2}|\mathcal{O}|.
    \]
    Since $0<r<d$, $\lambda_1\not=\lambda_2$.  For any $i\in[d]$, if we add the traces of all the elements of $\mathcal{O}$ which contain $x_i$, we obtain the following equalities
    $$
    \lambda_1\Sigma(P)=\lambda_1 x_i+\lambda_2\sum_{j\not=i}x_j=(\lambda_1-\lambda_2)x_i+\lambda_2\Sigma(X_{y^\ast}).
    $$
    Hence, $(\lambda_1-\lambda_2)x_i$ is independent of the choice of $i$.  Since $\lambda_1\not=\lambda_2$, it follows that all $x_i$ are equal, a contradiction.  Hence, $\pi$ admits no proper and nonempty trace invariant subsets.
\end{proof}
In general, characterizing the branched covers which admit trace-invariant subsets cannot be done on purely the group-theoretic level, as seen in the following example.  Indeed, whether $\pi:X \to Y$ admits proper and nonempty trace-invariant subsets depends on an embedding of $X$ whereas the the monodromy group of $\pi$ does not. 

\begin{example}
    Consider the branched cover 
    \begin{align*} 
    \pi: \mathbb{C} &\to \mathbb{C} \\ 
    z &\mapsto z^4
    \end{align*}
    The monodromy group is the cyclic group $C_4$ and fibres are of the form $\pi^{-1}(y^\ast) = \{a,b,-a,-b\}$, listed in their cyclic order induced by $C_4$. Clearly $\{a,-a\}$ is trace-invariant. 
    
    Now consider 
    \begin{align*} 
    \pi': \mathcal V(x^2-y)&\to \mathbb{C} \\ 
    (x,y) &\mapsto x^4
    \end{align*}
    We observe that, $\mathcal V(x^2-y)$ is isomorphic to $\mathbb{C}$ via $\phi(t) = (t,t^2)$ and $\pi = \pi' \circ \phi$, so $\pi'$ and $\pi$ represent the same branched cover. Thus, $G_{\pi'}=C_4$ acts cyclically on a fibre of the form 
    \[
    \pi'^{-1}(y^\ast) = \{(a,a^2),(b,b^2),(-a,a^2),(-b,b^2)\}.
    \]
    The subset $\{(a,a^2),(-a,a^2)\}$ is not trace-invariant since $a^2\not=b^2$. 
\end{example}

The following result combines our generalized trace test with the restricted monodromy results of earlier sections. 
\begin{theorem}[Generalized Restricted Trace Test]
\label{thm:generalizedrestrictedtrace}
        Let $\pi:X \to Y$ be a branched cover, with $X \subseteq \mathbb{C}^n$, $y^\ast \in Y$ generic, $P$ a nonempty subset of $X_{y^{\ast}}$, $\psi: Y \to Z$ a $\pi$-generic dominant map, and $V$ the level set through $y^\ast$. Suppose $G_\pi$ is the full symmetric group and $G_{\pi_V}$ contains a transposition. Then $\Xi_{P,\pi_V}$ descends to a function on $\mathcal U \cap V$ if and only if $P=X_{y^{\ast}}$. 
\end{theorem}
\begin{proof}
    \Cref{thm:main_theorem} shows that since $G_{\pi_V}$ contains a transposition and $G_\pi$ is the full symmetric group, $G_{\pi_V}$ is the full symmetric group too. In particular, $G_{\pi_V}$ is one-swappable,  so by \Cref{lem:sufficient_for_no_TIS}, $\pi_V$ admits no proper and nonempty trace-invariant subsets.  By \Cref{prop:generalizedtracetest}, since $\pi_V$ admits no proper and nonempty trace-invariant subsets, the function $\Xi_{P,\pi_V}$ descends if and only if $P$ is the whole fibre $X_{y^\ast}$. 
\end{proof}

We end this section with a remark about a further generalization, which we do not explore here, as well as a list of one-swappable groups.
\begin{remark}
    The trace map $\Sigma$ may be replaced throughout by any function
    $g$ defined on finite subsets of $X$.  In this case, one defines
    \[
        \Xi_{P,g}([\tau])=g(\sigma_\tau(P)),
    \]
    and the notions of descent and $g$-invariance are defined analogously.
    The trace test considered here is the special case $g=\Sigma$.  See \cite{Som02} for some details in this direction.
\end{remark}

\begin{remark}
    By exhaustive computation in GAP, the one-swappable groups of degree $\leq 22$ other than the natural alternating groups and natural symmetric groups are 
\[
\begin{array}{c|l}
\text{degree} & \text{one-swappable groups} \\ \hline
5 & C_{5},\ D_{10} ,\ C_5 \rtimes C_4\\
6 & A_5, S_5 \quad \text{(using their representations in }S_6\text{)}\\
7 & D_{14},\ C_7 \rtimes C_6  \\
8 & PSL(3,2) \rtimes C_2 \\
9 & PSL(2,8), PSL(2,8) \rtimes C_3
\end{array}
\]
  Ara\'ujo, Mitchell, and Schneider classify the permutation groups
    with the \emph{universal transversal property} (see 
    \cite[Definition~2.3 and Theorem~2.7]{AMS11}).
    Apart from the natural alternating and symmetric groups, their
    exceptional list agrees with
    our computed list of one-swappable groups through degree $22$,
    with one exception: $D_{14}$ is one-swappable but does not have the
    universal transversal property. We leave it to the reader to determine all one-swappable groups and relate them to the universal transversal property. 
    \label{rem:one-swappable-groups}
\end{remark}

\section{Trace Tests for Sparse Polynomial Systems}
\label{sec:sparse_trace_test}
Using our results, we complete the sparse trace
test introduced in \cite{STT23}. To state our improvements, we first describe the setting. Fix $n\in \N$ and let $\mydef{\C[x]} \vcentcolon= \C[x_1^{\pm 1},\dots,x_n^{\pm 1}]$ be the Laurent polynomial ring in $n$ variables. The \mydefit{support} of a polynomial $f(x)= \sum c^*_{\alpha} x^\alpha = \sum c^*_{\alpha} x_1^{\alpha_1}\cdots x_n^{\alpha_n} \in \C[x]$ is the finite set \mydef{\textrm{supp}$(f)$} consisting of the exponents $\alpha\in\mathbb{Z}^n$ such that the corresponding coefficient $c^*_\alpha \in \C$ is nonzero.  For $m\in \N$, let $\mydef{\mathcal A_{\bullet}} = (\mydef{\mathcal A_1},\ldots,\mydef{\mathcal A_m})$, where each $\mathcal A_i \subseteq \mathbb{Z}^n$ is a nonempty finite set. A \mydefit{sparse (Laurent) polynomial system} supported on $\mathcal A_{\bullet}$ is a system of $m$ polynomials $\mydef{\mathcal F} = (\mydef{f_1},\ldots,\mydef{f_m})$ in $n$ variables $x=(x_1,\ldots,x_n)$ where $\textrm{supp}(f_i) \subseteq \mathcal A_i$ for each $i$. When $m=n$, $\cF$ is called \mydefit{square}.

We denote the space of (coefficients of) polynomial systems supported on $\mathcal A_\bullet$ by \[\mydef{\C^\Abullet}\vcentcolon= \mathbb{C}^{|\mathcal A_1|} \times \cdots \times \mathbb{C}^{|\mathcal A_m|}\]  and  the polynomial system corresponding to $\mydef{c^\ast} =((c_{\alpha}^*)_{\alpha \in \mathcal A_{i}})_{i=1}^m\in \C^\Abullet$ by
\[\mydef{\mathcal F_{c^\ast}(x)} \vcentcolon= (f_1(x;c^\ast),\ldots,f_m(x;c^\ast)) = (f_1,\ldots,f_m) \in (\mathbb{C}[x])^m.
\]
Using indeterminates $c=((c_{\alpha})_{\alpha \in \mathcal A_i})_{i=1}^m$ in place of coefficients gives a system $\mydef{\mathcal F_{c}(x)} \in (\C[c][x])^m$. We call this the \mydefit{universal sparse polynomial system} over $\mathcal A_{\bullet}$.
For  \textit{essential} $\mathcal A_{\bullet}$, the incidence variety  consisting of zeros of $\mathcal F_{c}(x)$ (see Equation \eqref{eq:sparseincidence}) forms a branched cover $\pi_{\mathcal A_{\bullet}}$ over the space of coefficients.  This branched cover is the central object of this section.  Obtaining a trace test for $\pi_{\mathcal A_\bullet}$ was the main goal of \cite{STT23}.  

Like in the classical setting, the \textit{sparse trace test} \cite[Algorithm 1]{STT23} involves a restriction. The purpose of the restriction is to control the trace function and obtain an effective trace test.  In the sparse setting, the generic level sets are given by fixing some coefficients in $\mathbb{C}^{\mathcal A_{\bullet}}$, while varying others. In the original sparse trace test, the authors demand that 
\begin{itemize}
    \item $G_{\pi_{\mathcal A_{\bullet}}}$ is the full symmetric group,
    \item The fixed coefficients are generic,
    \item The coefficients being moved are \textit{abundant}, and
    \item The coefficients being moved are \textit{trace affine linear}.
\end{itemize}
The purpose of abundance is to guarantee that the restricted branched cover retains the full symmetric monodromy of $G_{\pi_{\mathcal A_{\bullet}}}$. This proves one-swappability and is used in the \textit{converse} of the sparse trace test. The purpose of trace affine linearity is to guarantee that the relevant trace function is linear, and thus can be checked easily in practice. 

Our contribution adjusts these assumptions: abundance is removed and trace affine linearity is replaced with descent of the local trace function. We remove abundance by applying our main theorem (\Cref{thm:main_theorem}) to the generic level set obtained by moving only one coefficient to obtain \Cref{cor:sparse-restriction-full-symmetric}. Under the assumption that $G_{\pi_{\mathcal A_{\bullet}}} = S_d$, the transposition in that restriction guarantees that full symmetric monodromy is preserved. Our general framework then replaces trace affine linearity with the weaker requirement that the trace descends to a function of the moving coefficient. Linearity is useful for making this condition easy to check, and the results of \cite[Lemma 10]{STT23} remain useful for identifying when the trace depends linearly on individual coefficients, but linearity is not necessary for the trace test itself. Bounds for the degrees of the trace in terms of the combinatorics of the support are given in \cite[Theorem 25]{STT23}. The result is a new, complete, sparse trace test, \Cref{alg:sparse-trace-test}.
 
\subsection{Structure of a set of supports} 
Consider a collection $\mathcal A_{\bullet}=(\mathcal A_1,\ldots,\mathcal A_m)$ of supports in $\mathbb{Z}^n$. For $\emptyset \neq I \subseteq [m]$, we define $\mydef{\mathcal{A}_I}$ as $(\mathcal{A}_i)_{i\in I}$ and write $\mydef{\dim(\mathcal A_I)}$ for the dimension of the affine~span, which equals the dimension of 
\[\mydef{\mathbb{R}\mathcal A_I} \vcentcolon= \textrm{span}_{\mathbb R}\{\alpha-\beta \mid \alpha,\beta \in \mathcal A_i \text{ for }i \in I\}.
\] We define the \mydefit{defect} of $\mathcal A_I$ to be $\dim(\mathcal A_I)-|I|$. 

Now suppose that $n=m$, that is, the set of supports is \textit{square}. 
We define the \mydefit{mixed volume} of $\Abullet$ to be the mixed volume of the convex hulls of the supports: \[\mydef{d_{\mathcal A_{\bullet}}} :\textrm{MV}(\Abullet) \vcentcolon= \textrm{MV}(\textrm{conv}(\mathcal A_1),\ldots,\textrm{conv}(\mathcal A_n)).
\] A square set of supports is said to be \mydefit{essential} if the defect of any $\mathcal{A}_I$ is nonnegative for all $I$.  The following result from Minkowski relates the essential property with the mixed volume.

\begin{lemma}[see \cite{Khov16}]
    The mixed volume of  $\mathcal A_{\bullet}$ is nonzero if and only if $\mathcal A_{\bullet}$ is essential. 
\end{lemma}

Next we discuss two special features that a square set of supports may have. We define the sublattice of $\mathbb{Z}^n$ generated by differences in some $\mathcal A_I$ as
\[
\mydef{\mathbb{Z}\mathcal A_I} \vcentcolon= \textrm{span}_{\mathbb Z}\{\alpha -\beta \mid \alpha,\beta \in \mathcal A_i \text{ for some }i \in I\}.
\] We write $\mydef{[\mathbb{Z}^n : \mathbb{Z}\mathcal A_I]}$ for the index of this lattice, if it is finite.

A set of supports $\mathcal A_{\bullet}$ is said to be \mydefit{triangular} if it has a proper subset $\mathcal A_I$ of defect zero.  We say that a square set of supports $\mathcal A_\bullet$ is \mydefit{strictly triangular} if $1<\textrm{MV}(\mathcal A_I)<\MV(\mathcal A_{\bullet})$ where $\mydef{\MV(\mathcal A_I)}$ denotes the mixed volume of the supports $\mathcal A_I$ of defect zero with respect to the $\dim(\mathcal A_I)$-dimensional volume form on their affine span. The support $\mathcal A_{\bullet}$ is \mydefit{lacunary} if the lattice index $[\mathbb{Z}^n:\mathcal L(\mathcal A_{\bullet})]$ is greater than one,  and a square set of supports is \mydefit{strictly lacunary} if $1<[\mathbb{Z}^n:\mathcal L(\mathcal A_{\bullet})]<d_{\mathcal A_{\bullet}}$. 

\subsection{Sparse systems as branched covers}
The \mydefit{incidence variety} of the universal sparse polynomial system with support $\mathcal{A}_{\bullet}$ is the variety of (toroidal) solution-coefficient pairs of $\mathcal F_{c}(x)$:
\begin{equation}    \label{eq:sparseincidence}
\mydef{X_{\mathcal A_{\bullet}}} \vcentcolon= \mathcal V^{\times}(\mathcal F_{c}(x)) = \{(x^*,c^\ast) \in (\mathbb{C}^\times)^n \times \mathbb{C}^{\mathcal A_{\bullet}} \mid \mathcal F_{c^\ast}(x^*) = \textbf{0}\} \subseteq (\mathbb{C}^\times)^n \times \mathbb{C}^{\mathcal A_{\bullet}}.
\end{equation}  We observe that $X_{\mathcal A_\bullet}$ is irreducible since the fibres of the projection onto the first factor consist of linear spaces of dimension $|\mathcal A_\bullet|-n$ over $(\mathbb{C}^\ast)^n$.  Define the map $\mydef{\pi_{\mathcal A_{\bullet}}}: X_{\mathcal A_{\bullet}} \to \mathbb{C}^{\mathcal A_{\bullet}}$ to be the projection onto the second factor. The fibre $\pi_{\mathcal A_{\bullet}}^{-1}(c^\ast)$ consists of the 
solutions in the \mydefit{torus} $\mathbb{C}^\times$ to the particular sparse system $\mathcal F_{c^\ast}(x)$. The Bernstein-Kushnirenko-Khovanskii (BKK) theorem tells us that $\pi_{\mathcal A_{\bullet}}$ is a branched cover of degree $d_{\mathcal A_{\bullet}} = \textrm{MV}(\Abullet)$.

\begin{lemma}[BKK, \cite{BKK75, Kush75, Khov78}]\label{lem:BKK}
    A sparse polynomial system with square support $\mathcal A_{\bullet}$ has at most $d_{\mathcal A_{\bullet}}$-many isolated solutions over $\C^\times$. For a generic system supported on $\mathcal A_{\bullet}$, this is attained and all solutions are isolated and regular.
\end{lemma}

The notion of generic in \Cref{lem:BKK} agrees with the notion of generic in the definition of a branched cover. Over the branch locus $\Delta$, fibres have fewer than $d_{\mathcal{A}_\bullet}$ regular isolated points.  In particular, fibres over $\Delta$ may have singular points, solutions outside the torus, or positive-dimensional solution sets. Esterov proved conditions on $\Abullet$ under which the monodromy group $G_{\pi_\Abullet}$ is the full symmetric group $S_{d_{\mathcal{A}_\bullet}}$.  

\begin{lemma}[\cite{Est19}, see also \cite{Sott25, STT23}]
\label{lem:est-uniform-monodromy}
    Let $\mathcal A_{\bullet}$ be a square set of supports. If $d_\Abullet = 2$, or $\Abullet$ is non-lacunary and non-triangular, then $G_{\pi_{\mathcal{A}_\bullet}}$ is the full symmetric group $S_{d_{\mathcal{A}_\bullet}}$. Otherwise $G_{\pi_\Abullet}$ is not the full symmetric group.
\end{lemma}

\subsection{Restrictions of sparse polynomial systems} 
We focus on restrictions of sparse branched covers in which exactly one coefficient  varies. Let $\mathcal A_{\bullet}=(\mathcal A_1,\ldots,\mathcal A_n)$ be an essential set of supports. Fix $\beta \in \mathcal A_{\bullet}$, meaning that $\beta\in\mathcal A_i$ for some $i\in[n]$. Suppose $\beta$ is the exponent of the monomial whose coefficient we vary and define
$$\mydef{\mathcal{A}_\bullet-\beta}\vcentcolon=(\mathcal A_1,\dots,\mathcal A_i-\{\beta\},\dots,\mathcal A_n).$$
Define the map $\mydef{\psi_\beta}:\mathbb{C}^{\mathcal A_\bullet}\rightarrow\mathbb{C}^{\mathcal A_\bullet-\beta}$ which forgets the coefficient of  $\textbf{x}^\beta = x_1^{\beta_1}\cdots x_n^{\beta_n}$ in the $i$-th sparse polynomial  $f_i$, thought of as an element of  $\mathbb{C}^{\mathcal A_i}$.  A generic level set of $\psi_\beta$ is a \mydefit{coordinate line}
    \[
V \vcentcolon= \mydef{V_{\beta,c^\ast}} \vcentcolon= \{c^\ast+te_\beta \mid t \in \mathbb{C}\} = \psi_{\beta}^{-1}(\psi_{\beta}(c^\ast)).
    \]
    Fix a generic $c^\ast\in\bC^{\cA_\bullet}$. Henceforth, we use $V$ for $V_{\beta,c^\ast}$ when doing so does not cause any confusion. 

Our goal is to use \Cref{thm:iff-full-symmetric} to study the restricted monodromy group of $V$.  We note $V$ is not generic in the sense of \cite{Zar29,NumericalGalois18} since it is parallel to a coordinate axis and so those results on fundamental groups of restrictions to generic lines do not apply. 

As usual, we consider the restriction ${\pi_{V}}:X_V\rightarrow V$. Then 
 \Cref{prop:restricted-is-branched-cover} and \Cref{lem:existence-of-transp-STT} verify the two conditions of \Cref{thm:iff-full-symmetric}: that $\psi_{\beta}$ is $\pi_{\mathcal A_{\bullet}}$-generic and that $G_{\pi_V}$ contains a transposition. Each makes assumptions on $\mathcal A_{\bullet}$.

\begin{proposition}
\label{prop:restricted-is-branched-cover}
Suppose $\mathcal A_{\bullet}$ is essential, $\beta \in \mathcal A_i$ for some $i \in [n]$, and  \[\mathcal A'_{\bullet} = (\mathcal A_1,\ldots,\widehat{\mathcal A_i},\ldots,\mathcal A_n)\] is non-triangular. Then  the restriction of $\pi_{\mathcal A_{\bullet}}$ to a generic level set of $\psi_{\beta}$ is a branched cover of degree $d_{\mathcal A_{\bullet}}$.
\end{proposition}

\begin{proof}
    By \Cref{lem:BKK} and the genericity of $c^\ast$, dominance is automatic, so it suffices to show that $X_{V}$ is irreducible for $V=V_{\beta,c^\ast}$.  Let $\mathcal{F}_{c^\ast}=(f_1,\dots,f_n)$ be the system corresponding to $c^\ast$.  First, we show that $X_{V}$ is a graph over the variety $\mathcal{V}^\times(f_1,\ldots,\widehat{f_{i}},\ldots,f_n)$, where $\widehat{f_i}$ indicates that this polynomial is omitted from the sequence. Indeed, for any $x^\ast\in \mathcal{V}^\times(f_1,\ldots,\widehat{f_{i}},\ldots,f_n)$, there is exactly one value for $t$ which makes $(x^\ast,c^\ast+te_\beta)\in X_{V}$, namely $t=-f_i(x^\ast)/(x^\ast)^\beta$.
    
     Next, we show that $X_{V}$ is irreducible. Since $X_V$ is a graph over $\mathcal{V}^\times(f_1,\ldots,\widehat{f_{i}},\ldots,f_n)$, it suffices to show that the latter variety is irreducible. Since $\mathcal A'_{\bullet}$ is assumed to be non-triangular,  $\mathcal{V}^\times(f_1,\ldots,\widehat{f_{i}},\ldots,f_n)$ is irreducible by Khovanskii's irreducibility theorem \cite[Theorem 17]{Khov16}:  a variety defined by a generic underdetermined sparse polynomial system with non-triangular support is irreducible.
\end{proof}

We note that if $\mathcal A_{\bullet}$ is non-triangular, then every possible $\mathcal A_{\bullet}'$ occurring in \Cref{prop:restricted-is-branched-cover} is non-triangular as well.

\begin{lemma}[{\cite[Theorem 39]{STT23}}]
    \label{lem:existence-of-transp-STT}
    If $\Abullet$ is a non-triangular, non-lacunary set of supports with $\MV(\Abullet)\ge 2$, and $\beta \in \mathcal A_{\bullet}$, then the  monodromy group $G_{\pi_{V}}$ of the restriction of $\pi_{\mathcal A_{\bullet}}$ to a generic level set of $\psi_{\beta}$ contains a transposition. 
\end{lemma}
\begin{proof}
In the proof of \cite[Theorem 39]{STT23}, the authors show that a generic coordinate line meets the $\Abullet$-discriminant transversely at a point above which $\pi_\Abullet$ is simply ramified.  Moreover, the $\Abullet$-discriminant is contained within an irreducible component of $\Delta$ of codimension 1.  Since $\mathbb{C}^{\mathcal{A}_\bullet}$ and $V$ are smooth varieties, \Cref{cor:restricted-mon-gp-transposition} implies that $G_{\pi_{V}}$ contains a transposition.
\end{proof}

\begin{remark}
    The hypothesis of \Cref{lem:existence-of-transp-STT} may be relaxed from non-lacunary and non-triangular to being non-lacunary and \emph{dual effective}. 
    Esterov calls a tuple of supports \mydefit{dual effective} when its
    $\Abullet$-discriminant has a codimension-1 component \cite[Definition~3.14]{Est19}.  Moreover, \cite[Theorem~3.25]{Est19} describes a generic point of this discriminant as having a unique double root and shows that monodromy around a transverse curve induces a transposition. We are unaware of a clean combinatorial characterization of dual effectiveness. However, \cite[Corollary 3.23]{Est19} states a simple combinatorial condition which implies it: being non-lacunary, non-triangular, and having mixed volume at least two. This motivates our use of a simpler condition in \Cref{lem:existence-of-transp-STT}. For additional details and a discussion of these techniques, see also \cite[Theorem~4.12]{dissertation}.
\end{remark}

\begin{theorem}
\label{cor:sparse-restriction-full-symmetric}
    Let $\mathcal A_{\bullet}$ be  non-lacunary and non-triangular with $d_{\mathcal A_{\bullet}}\geq 2$ and let $V$ be a generic level set of $\psi_\beta$ for some $\beta \in \mathcal A_{\bullet}$. Then the monodromy group $G_{\pi_V}$ is the full symmetric group.
\end{theorem}

\begin{proof}
By \Cref{lem:est-uniform-monodromy}, the monodromy group $G_{\pi_{\mathcal A_\bullet}}$ is the full symmetric group.  Since the coordinate line $V$ is a generic level set of the map $\psi_\beta$, \Cref{prop:restricted-is-branched-cover} implies that the restriction $\pi_{V}$ is a branched cover of degree $d_{\mathcal A_{\bullet}}$. \Cref{lem:existence-of-transp-STT} implies that $G_{\pi_{V}}$ contains a transposition, so by \Cref{cor:uniformorimprimitive}, the restricted monodromy group is a full wreath product. Finally, applying \Cref{thm:iff-full-symmetric}, we conclude that $G_{\pi_{V}}$ is also the full symmetric group.
\end{proof}

\subsection{Completion of the sparse trace test}\label{sec:completionsparsetrace}
Following the construction in \Cref{sec:generalizedtrace}, we let $c^\ast\in\C^{\cA_\bullet}-\Delta$, $\emptyset \neq P\subseteq \pi_{\mathcal A_{\bullet}}^{-1}(c^*)$, and 
    \begin{align*}
    {\Xi_P}: \Pi_1(\mathbb{C}^{\mathcal A_\bullet}-\Delta, c^\ast, \cdot) &\to \C^n \\
    [\tau] &\mapsto \Sigma(\sigma_{\tau}(P)).
    \end{align*}
Next, we fix $\beta \in \mathcal A_{\bullet}$ and study ${\Xi_{P,\pi_V}}$ where $V$ is the generic level set of $\psi_\beta$ through $c^\ast$  so that 
$$\Xi_{P,\pi_V}:\Pi_1(V-\Delta, c^\ast, \cdot) \to \C^n.$$
By \Cref{cor:sparse-restriction-full-symmetric} and
\Cref{thm:generalizedrestrictedtrace}, we complete the theoretical foundation of the sparse
trace test.

\begin{theorem}[Sparse trace test]
    \label{thm:sparse-trace-test}
    Let $\Abullet$ be a non-lacunary and non-triangular set of supports with $d_{\mathcal A_{\bullet}}\geq 2$. Let $c^\ast\in\mathbb{C}^{\mathcal A_\bullet}$ be generic and $\emptyset \neq  P\subseteq\pi_{\mathcal A_{\bullet}}^{-1}(c^*)$.  Fix $\beta\in\mathcal{A}_{\bullet}$.  Then  $ P=\pi_{\mathcal A_{\bullet}}^{-1}(c^*)$ if and only if $\Xi_{P,\pi_V}$ descends to a function on $V-\Delta$.
\end{theorem}
\begin{corollary}
\label{cor:sparse-trace-test-first-coordinate}
    Under the hypotheses of \Cref{thm:sparse-trace-test}, let
    $\Xi_{P,\pi_V} = (\Xi^{(1)}_{P,\pi_V}, \ldots,\Xi_{P,\pi_V}^{(n)})$. Then
    \[
        P=\pi_{\mathcal A_\bullet}^{-1}(c^\ast)
        \iff
        \Xi^{(i)}_{P,\pi_V}
        \text{ descends to a function on }V-\Delta\text{ for any }i\in[n].
    \]
\end{corollary}

\begin{proof}
    Let $\Sigma=(\Sigma^{(1)},\dots,\Sigma^{(n)})$.
    By the generalized trace test applied to $\Sigma^{(1)}$,
    $\Xi^{(1)}_{P,\pi_V}$ descends if and only if $P$ is
    $\Sigma^{(1)}$-trace-invariant, meaning \[
    \Sigma^{(1)}(g(P))=\Sigma^{(1)}(P)
    \qquad
    \text{for all }g\in G_{\pi,y^\ast}.
\]

    By \Cref{cor:sparse-restriction-full-symmetric},
    $G_{\pi_V}$ is the full symmetric group, and by
    \cite[Lemma~42]{STT23}, the first coordinates of the points in a
    generic fibre are distinct. Consequently, $\pi_V$ admits no proper
    $\Sigma^{(1)}$-trace-invariant subsets. Hence, the only nonempty
    subset for which $\Xi^{(1)}_{P,\pi_V}$ descends is the entire fibre.
\end{proof}

The main consequence of \Cref{thm:sparse-trace-test} is that completeness of the solution set to a sparse polynomial system (satisfying the above hypotheses) can be checked by changing \textit{any} coefficient and observing the behavior of the zeros. That any coefficient may be used theoretically \textit{completes} the sparse trace test \cite[Algorithm 1]{STT23}, which previously imposed conditions on the moving coefficients. An effective version of \Cref{thm:sparse-trace-test} requires the ability to ascertain whether $\Xi_{P,\pi_V}$ is a function on the line $V$. The sparse trace test (\Cref{thm:sparse-trace-test}) can be brought closer to the classical trace test (\Cref{prop:classicaltracetest} and \cite{Som02,Ley18}) by choosing $\beta$ to satisfy conditions \cite[Lemma 10]{STT23} which give bounds on the degree of $\Xi_{P,\pi_V}$.  Once degree bounds are fixed, rational interpolation becomes an effective approach.

\begin{example}\label{ex:trace}
    Let $\mathcal{F}_{c} = (f_{c},g_{c})$ be the family of bivariate sparse polynomial systems supported on $cA_f=\{(0,0),(1,0),(0,1),(1,1),(2,1),(1,2)\}$ and $\cA_g=\{(0,0),(1,0),(0,1),(1,1)\}$.  By \Cref{lem:BKK}, for a generic choice of coefficients, this system has four solutions.  We focus on the restricted branched cover $X_{\beta,c^\ast}$ corresponding to $f_{c^\ast} = -1 + 3x - 3y + xy + x^2y - 2xy^2$ and $g_{c^\ast}=1 + x + y - xy$ with $\beta=(1,0)\in\cA_g$.  Via a Gr\"obner basis calculation, we observe that the trace of the $x$- and $y$-coordinates of complete solution sets are functions of the coefficient $c_\beta$, as follows
        \[
        \Sigma_x(\mathcal{F}_{c^\ast+(c_\beta-c_\beta^\ast)e_\beta}) = 2c_\beta - \frac{4}{c_\beta} \hspace{3cm}
        \Sigma_y(\mathcal{F}_{c^\ast+(c_\beta-c_\beta^\ast)e_\beta}) = 2c_\beta - \frac{3}{2}.
    \]
    See \Cref{fig:sparsetracetest} for additional details.

    \begin{figure}[ht]
    \centering
    \begin{subfigure}{0.3\textwidth}
        \centering   \includegraphics[width=\linewidth]{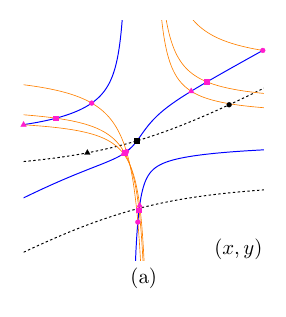}
    \end{subfigure}
    \begin{subfigure}{0.3\textwidth}
        \centering   \includegraphics[width=\linewidth]{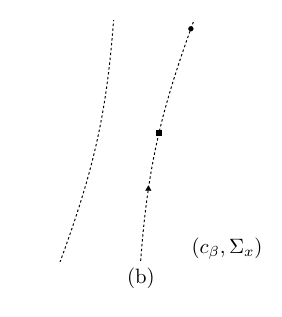}
    \end{subfigure}
    \begin{subfigure}{0.3\textwidth}
        \centering   \includegraphics[width=\linewidth]{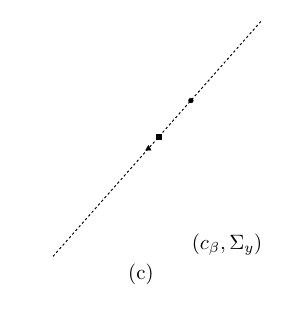}
    \end{subfigure}
    \caption{The sparse trace test for the system $\mathcal{F}_{c^\ast+(c_\beta-c_\beta^\ast)e_\beta}=(f_{c^\ast},g_{c^\ast+(c_\beta-c_\beta^\ast)e_\beta})=(-1 + 3x - 3y + xy + x^2y - 2xy^2, 1 + c_\beta x + y - xy)$, as in \Cref{ex:trace}.  In (a), $f_{c^\ast}=0$ is depicted in the thicker blue curves and three curves for $g_{c^\ast+(c_\beta-c_\beta^\ast)e_\beta}=0$ for different values of $c_\beta$ are displayed in the thinner orange curves.  The intersections of the curves are marked in magenta, where the shape of each marker corresponds to a value of $c_\beta$.  The traces of these quadruples of points appear on the dashed curve, with matching marker in black.  The relationship between $c_\beta$ and $\Sigma_x$ or $\Sigma_y$ is illustrated in (b) and (c), respectively.  By inspection, we see that both $\Sigma_x$ and $\Sigma_y$ are functions of $c_\beta$.}
    \label{fig:sparsetracetest}
\end{figure}

On the other hand, the trace of the $x$- and $y$-coordinates on triples of solutions (that is, incomplete solution sets) is not a function of the coefficient $c_\beta$.  In particular, briefly let $\Sigma_x$ be the $x$-coordinate of the sum of three distinct solutions to $\mathcal{F}_{c^\ast+(c_\beta-c_\beta^\ast)e_\beta}$ and $\Sigma_y$ be the $y$-coordinate of the sum of these three solutions.  Then $\Sigma_x$ and $c_\beta$ are related by the vanishing of the following polynomial
\begin{multline*}
    8c_\beta^6\Sigma_x-12c_\beta^5\Sigma_x^2+6c_\beta^4\Sigma_x^3-c_\beta^3\Sigma_x^4+32c_\beta^5-80c_\beta^4\Sigma_x+56c_\beta^3\Sigma_x^2-12c_\beta^2\Sigma_x^3+22c_\beta^4-25c_\beta^3\Sigma_x\\
    +7c_\beta^2\Sigma_x^2
    -126c_\beta^3 + 159c_\beta^2\Sigma_x-48c_\beta \Sigma_x^2-102c_\beta^2+56c_\beta \Sigma_x+124c_\beta-64\Sigma_x+112.
\end{multline*}
Similarly, $\Sigma_y$ and $c_\beta$ are related by the vanishing of the following polynomial
\begin{multline*}
8c_\beta^3\Sigma_y-24c_\beta^2\Sigma_y^2+24c_\beta \Sigma_y^3-8\Sigma_y^4+16c_\beta^3-80c_\beta^2\Sigma_y+100c_\beta \Sigma_y^2\\-36\Sigma_y^3-58c_\beta^2+106c_\beta \Sigma_y-34\Sigma_y^2+19c_\beta+17\Sigma_y+21.
\end{multline*}
Since these polynomials are square-free and not linear in $\Sigma_x$ or $\Sigma_y$, respectively, this indicates that neither $\Sigma_x$ nor $\Sigma_y$ is a function of $c$.  See \Cref{fig:sparsetracetest2} for additional evidence that these are not functions of $c_\beta$.

\begin{figure}[ht]
    \centering
    \begin{subfigure}{0.3\textwidth}
        \centering   \includegraphics[width=\linewidth]{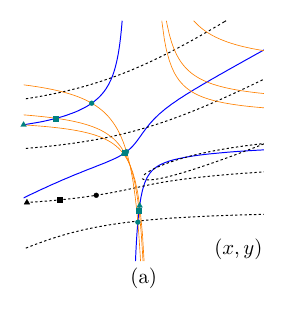}
    \end{subfigure}
    \begin{subfigure}{0.3\textwidth}
        \centering   \includegraphics[width=\linewidth]{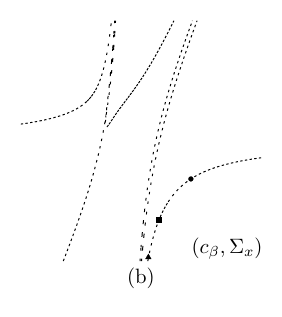}
    \end{subfigure}
    \begin{subfigure}{0.3\textwidth}
        \centering   \includegraphics[width=\linewidth]{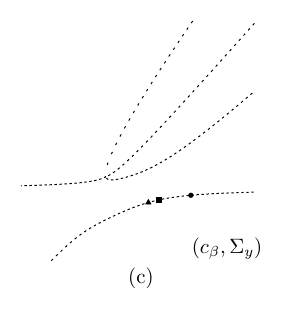}
    \end{subfigure}
    \caption{A depiction of the sparse trace test for an incomplete solution set of the system $\mathcal{F}_{c^\ast+(c_\beta-c_\beta^\ast)e_\beta}=(f_{c^\ast},g_{c^\ast+(c_\beta-c_\beta^\ast)e_\beta})=(-1 + 3x - 3y + xy + x^2y - 2xy^2, 1 + c_\beta x + y - xy)$, as in \Cref{ex:trace}.  In (a), $f_{c^\ast}=0$ is depicted in the thicker blue curves and three curves for $g_{c^\ast+(c_\beta-c_\beta^\ast)e_\beta}=0$ for different values of $c_\beta$ are displayed in the thinner orange curves.  Incomplete sets consisting of three intersection points are marked in teal, where the shape of each marker corresponds to a value of $c_\beta$.  The traces of these triples of points appear on the dashed curve, with matching marker in black.  The relationship between $c_\beta$ and $\Sigma_x$ or $\Sigma_y$, respectively, is illustrated in (b) and (c).  By inspection, we see that neither $\Sigma_x$ nor $\Sigma_y$ is a function of $c_\beta$.}
    \label{fig:sparsetracetest2}
\end{figure}

\end{example}

We propose an algorithm, \Cref{alg:sparse-trace-test}, as a version of the sparse trace test using interpolation. The algorithm requires prior knowledge of the maximum degree dependence of the trace on $c_\beta$, which can be found as follows. When the support element $\beta$, whose coefficient $c_\beta$ is allowed to vary, is in the interior of the convex hull of the support, the trace as a function of $c_\beta$ is a polynomial with bounded degree, see \cite[Theorem 25, Remark 26]{STT23} for specific degree bounds. When $\beta$ is on the boundary of the convex hull, the trace is a rational function, but degree bounds for the numerator and denominator are similarly guaranteed. By \cite[Chapter 8, Proposition 1.6]{GKZ94}, the degree of the resultant is bounded by a mixed volume calculation. Adapting this result to a hidden variable resultant, as is used to write the trace function in \cite{STT23}, requires an additional mixed volume factor. 
	\begin{algorithm}
		\caption{Sparse trace test (interpolation)}
		\label{alg:sparse-trace-test}
		\begin{algorithmic}[1]
			\Require
			\Bullet{$\Abullet$ in $\Z^n$, an essential, non-lacunary, non-triangular collection of supports}
            
			\Bullet{ $\mathcal{F} \in \C^{\Abullet}$, a generic system with coefficients $c^\ast$}
			\Bullet{$\emptyset \neq {P} \subseteq \V^\times(\mathcal{F})$}
			
            \Bullet{$\beta\in \Abullet$}
            \Bullet{$N$ and $D$, maximum degrees for the numerator and denominator of the trace function $\Sigma^{(1)}$}
			\Ensure
			\Bullet{if ${P} = \V^\times(\mathcal{F})$, then \texttt{pass}, else \texttt{fail}}

            \State Choose generic $c_{\beta_1},\dots, c_{\beta_{N+D+2}}\in \C$.

            \State Define the paths $[\tau_i]\in
\Pi_1(V-\Delta,c^\ast,c^\ast+(c_{\beta_i}-c_\beta^\ast)e_\beta)$ as $\tau_i \vcentcolon= c^*+t(c_{\beta_i}-c_\beta^\ast)e_\beta$.
            
			\State Use homotopy continuation to follow lifts of $\tau_i$ from ${P}$ to $\sigma_{\tau_i}(P)$ (see \Cref{sec:generalizedtrace}). 
            
			\State Compute $\Sigma^{(1)}(\sigma_{\tau_i}(P))$ for all $i$.
            
			\If{there exists a  rational function of numerator/denominator degrees at most $N$ and $D$ respectively  which fits $\{(c_{\beta_i},\Sigma^{(1)}(\sigma_{\tau_i}(P)))\}$,}
				\State \Return \texttt{pass}
			\Else
				\State \Return \texttt{fail}
			\EndIf
		\end{algorithmic}
	\end{algorithm}
\begin{theorem}
    \Cref{alg:sparse-trace-test} returns \texttt{pass} if and only if
    $P$ is complete.
    \label{thm:alg-1-correct}
\end{theorem}
\begin{proof}
    Suppose first that $P$ is complete. By \Cref{cor:sparse-trace-test-first-coordinate}, $\Xi^{(1)}_{P,\pi_V}$ descends to a function on $V-\Delta$.
    By the assumed degree bounds, this function is rational in $c_\beta$ with numerator degree at most $N$ and denominator degree at most $D$.  Therefore, the trace values computed in \Cref{alg:sparse-trace-test} fit such a rational function, and the algorithm returns \texttt{pass}.

    Conversely, suppose that the algorithm returns \texttt{pass}. Then the $N+D+2$ generic trace evaluations fit a rational function with numerator degree at most $N$ and denominator degree at most $D$. By the genericity of the interpolation points, this implies that $\Xi^{(1)}_{P,\pi_V}$ descends to such a rational function on $V-\Delta$. By \Cref{cor:sparse-trace-test-first-coordinate}, $P$ is complete.
\end{proof}

\appendix

\section{Fibre powers}\label{sec:fiberbranched} 
We provide explicit constructions to illustrate the relationship between fibre powers and decomposable branched covers.  In particular, our goal is to give a constructive version of the nonconstructive statement in \Cref{lem:imprimitive_iff_decomposable}.  Although the results in this section are not required for understanding the rest of this paper, the explicit constructions in this section provide an alternate view of our results.

Suppose that $\pi:X\rightarrow Y$ is a branched cover of degree $d$, then the \mydefit{$l$-th fibre power}
\[
\mydef{X^{(l)}} \vcentcolon= \{(x_{i_1},\ldots,x_{i_l};y) \mid x_{i_j} \in \pi^{-1}(y) \text{ for all }j=1,\ldots,l\}
\]
consists of  ordered $l$-tuples of elements in a fibre over a common base point.  The map $\pi^{(l)}: X^{(l)} \to Y$ is the projection onto the last coordinate, and it is dominant of finite degree $d^l$. However, the map $\pi^{(l)}$ is not a branched cover since $X^{(l)}$ is not irreducible. In fact, there are several components contained in the \mydefit{big diagonal}
\[
\mydef{\textrm{diag}\left(X^{(l)}\right)} \vcentcolon= \{(x_{i_1},\ldots,x_{i_l};y) \mid i_j  = i_{j'} \text{ for some }j\neq j'\} \subseteq X^{(l)}.
\]
We focus on components of $X^{(l)}$ which are not contained in the big diagonal.  In the following constructions, we slightly shrink $\cU$ to avoid singularities of $X^{(l)}$.  In particular, we define
\[
\mydef{\cU^{(l)}}\vcentcolon=\cU-\pi^{(l)}(\textrm{sing}(X^{(l)})),
\]
where $\textrm{sing}(X^{(l)})$ denotes the singular subvariety of $X^{(l)}$.  Since $X^{(l)}$ and $Y$ are the same dimension and the singular subvariety is a proper subvariety, $\cU^{(l)}$ is nonempty.  Similarly, we define $\mydef{\Delta^{(l)}}\vcentcolon=Y-\cU^{(l)}$.  With this setup, any point $x_I = (x_{i_1},\ldots,x_{i_l})$ of $l$ distinct elements of the fibre $\pi^{-1}(u)$ with $u\in\cU^{(l)}$ is necessarily in a unique irreducible component of $X^{(l)}$ that is not contained in the big diagonal.

Given a path $\tau \in \Pi_1(\mathcal U^{(l)},u,u')$ and an ordered subset $x_I=(x_{i_1},\dots,x_{i_l})$ of points in $X_u$, there are two natural ways to interpret
\[
\sigma_\tau(x_I).
\]
The first way is to lift $\tau$ to a path $x_I(t)$ in $X^{(l)}$ which starts at $x_I$.  Then, we may interpret $\sigma_\tau(x_I)$ as $x_I(1)$.  Since $\tau$ is a path in $\mathcal U^{(1)}$, this lift is uniquely defined and remains in an irreducible component of $X^{(l)}$.  The second way is to lift $\tau$ to the $l$ paths $\{x_{i_1}(t),\dots,x_{i_l}(t)\}$ such that $x_{i_j}(0)=x_{i_j}$.  Then, we may interpret $\sigma_\tau(x_I)$ as $(x_{i_1}(1),\dots,x_{i_l}(1);u')$.  Briefly, we define $\mydef{\pi_j}:X^{(l)}\rightarrow X$ to be the projection map from points of $X^{(l)}$ to their $j$-th entry. Then, $\pi_j(x_I(t))$ is a path in $X$ which starts at $x_{i_j}$ and maps onto $\tau$ by $\pi$. By uniqueness of lifts, we must have that $\pi_j(x_I(t))=x_{i_j}(t)$.  Hence, the two interpretations of $\sigma_\tau(x_I)$ coincide and may be used interchangeably.

\begin{lemma}
\label{lem:groupoid_on_fibre_power}
Let $\pi:X \to Y$ be a branched cover of degree $d$ and $\hat{X}$ be a component of $X^{(l)}$ not supported on the big diagonal such that $\pi^{(l)}(\hat{X})\not\subseteq \Delta^{(l)}$. The map $\pi^{(l)}:\hat{X} \to Y$ is a branched cover, and for any pair of points $(x_I;u)$ and $(x_J;u')$ in $\hat{X}\cap\left(\pi^{(l)}\right)^{-1}(\cU^{(l)})$, there exists $[\tau]\in \Pi_1(\mathcal U^{(l)},u,u')$ so that $\sigma_{\tau}(x_I) = x_J$, that is, $\sigma_{\tau}(x_{i_s}) = x_{j_s}$ for all $s=1,\ldots,l$.
\end{lemma}

 \begin{proof}
Since $\hat{X}$ is irreducible, we must show that $\pi^{(l)}:\hat{X} \to Y$ is dominant and of finite degree.  Let $x_I=(x_{i_1},\ldots,x_{i_l};u)\in\hat{X}\cap\left(\pi^{(l)}\right)^{-1}(\cU^{(l)})$.  Since $X^{(l)}$ is smooth at $x_I$, $x_I$ is in exactly one irreducible component, $\hat{X}$.  Moreover, since $\hat{X}$ is not supported on the diagonal, $x_I$ is not in the diagonal, that is the $x_{i_j}$'s are distinct.

Suppose that $x_J=(x_{j_1},\dots,x_{j_l})\in \hat{X}\cap\pi^{-1}(\cU^{(l)})$.  Since $X$ is irreducible, there is some path $\gamma$ in $\pi^{-1}(U^{(l)})$ connecting $x_{i_1}$ and $x_{j_1}$.  Let $\tau$ be the projection of $\gamma$ to $Y$ and $x_I(t)$ be the lift of $\tau$ to $X^{(l)}$ starting at $x_I$.  By construction, since $x_I(t)$ avoids the singular set of $X^{(l)}$, it remains in the single component $\hat{X}$.  Using the projection onto the first entry, we observe that the path $\pi_1(x_I(t))$ in $X$ is a lift of $\tau$ starting at $x_{i_1}$. By uniqueness of paths, $\pi_1(x_I(t))=\gamma(t)$, so $\pi_1(x_I(1))=x_{j_1}$.  Therefore, $\pi_1:\hat{X}\rightarrow X$ is dominant as its image includes the large open set $\pi^{-1}(\cU^{(l)})$.  Hence $\pi^{(l)}=\pi\circ\pi_1$ is dominant as it is the composition of dominant maps.  In addition, since $\pi^{(l)}(\hat{X})\supseteq \cU^{(l)}$, and over any point of $\cU^{(l)}$ there are $d^l$ preimages, $\pi^{(l)}:X\rightarrow Y$ is of finite degree.

To see that there is an element $[\tau]$ of the fundamental groupoid which induces the map $(x_I;u) \mapsto (x_J;u')$, take any path $\gamma$ in $\hat{X}\cap \left(\pi^{(l)}\right)^{-1}(\cU^{(l)})$ from $(x_I;u)$ to $ (x_J;u')$.  Then $\tau$ is the image of this path under $\pi^{(l)}$.
\end{proof}

We now use this lemma to study the behavior of block systems within fibre products.  In the following two results, we translate the properties of block systems to the case of branched covers.

\begin{corollary}\label{cor:blocks}
    Let $\pi:X\rightarrow Y$ be a branched cover of degree $d$, and assume that $G_\pi\lneq S_d$ contains a transposition.  Let $u,u'\in\cU$. Suppose that $B=\{x_{i_1},\dots,x_{i_k}\}$ is a block of the minimal block system of $G_{\pi,u}$, then 
    $$\{\sigma_\tau(B):[\tau]\in\Pi_1(\cU,u,u')\}$$
    is the minimal block system for $G_{\pi,u'}$.  In particular, if $x_I$ is an ordered subset of a block of $G_{\pi,u}$, then $\sigma_\tau(x_I)$ is an ordered subset of a block of $G_{\pi,u'}$.
\end{corollary}

Before beginning the proof, we note that \Cref{lem:YoungBlocks} and \Cref{cor:uniformorimprimitive,cor:uniqueness} together with the conditions of the statement imply that $G_\pi$ is imprimitive and has a unique minimal block system, guaranteeing the existence of $B$ in the statement of the Corollary.

\begin{proof}
We observe that since $G_{\pi,u}\cong G_{\pi,u'}$ via conjugation, $G_{\pi,u'}$ also has a unique minimal block system.  Suppose, towards a contradiction, that $\sigma_\tau(B)$ is not a block of the minimal block system of $G_{\pi,u'}$ for some $[\tau]\in\Pi_1(\cU,u,u')$.  Then, there are $s,s'\in\sigma_\tau(B)$ such that $s$ and $s'$ belong to different blocks of $\cB_{u'}$.  By \Cref{cor:uniformorimprimitive}, $G_\pi$ is a full wreath product, so there is a monodromy element $\sigma_\gamma\in G_{\pi,u'}$ that stabilizes $s'$ within its block, but takes $s$ to an element $\hat{s}$ not belonging to $\sigma_\tau(B)$.  Then $[\tau^{-1}\circ\gamma\circ\tau]\in\Pi_1(\cU,u)$, but $\sigma_{\tau^{-1}\circ\gamma\circ\tau}(B)$ is not a block of $G_\pi$.
\end{proof}

\begin{lemma}\label{lem:blockcomponent}  
    Let $\pi:X\rightarrow Y$ be a branched cover of degree $d$, and assume that $G_{\pi} \lneq S_d$ contains a transposition.  Let $u\in\cU$ be generic and $\mathcal B_u=\{B_1,\ldots,B_b\}$ be the unique minimal block system of $G_{\pi,u}$.  Then  there is a unique component $X_{\mathcal B}^{(k)}$ of $X^{(k)}$, where $k=d/b$, which contains all orderings of all minimal blocks of $G_{\pi,u}$.  
\end{lemma}
\begin{proof}
     We observe that since $\cU$ and $\cU^{(k)}$ differ in real dimension 2, the inclusion $\iota:\cU^{(k)}\rightarrow\cU$ induces a surjection $\iota_\ast:\Pi_1(\cU^{(k)},u)\twoheadrightarrow\Pi_1(\cU,u)$.  In particular, every permutation induced from $\Pi_1(\cU,u)$ has a representative in $\Pi_1(\cU^{(k)},u)$.  By genericity, we may assume that $u\in\cU^{(k)}$.

     By \Cref{cor:uniformorimprimitive}, $G_\pi$ is a full wreath product, that is, $G_\pi\cong S_k\wr H$. Therefore, for any two orderings $x_I$ and $x_J$ of $B_i$, there exists an element of the monodromy group taking one ordering to the other, while fixing everything else. By the surjectivity of $\iota_\ast$, this permutation can be realized by $[\gamma]\in\Pi_1(\cU^{(k)},u)$. Thus, by \Cref{cor:blocks}, $\sigma_\gamma(x_I) = x_J$.

     Due to the transitivity of $G_\pi$, for any $x_i\in B_i$ and $x_j\in B_j$, there is some element of the monodromy group taking $x_i$ to $x_j$.  By the surjectivity of $\iota_\ast$, this permutation can be realized by $[\tau]\in\Pi_1(\cU^{(k)},u)$, that is, $\sigma_\tau(x_i)=x_j$.  By \Cref{cor:blocks}, this must take the entire block $B_i$ to $B_j$, that is, $\sigma_\tau(B_i)=B_j$.

     Finally, since the lifts of $\gamma$ and $\tau$ avoid the singular set of $X^{(k)}$, their endpoints are in the same irreducible component of $X^{(k)}$.  On the other hand, the constructions of $\gamma$ and $\tau$, respectively, correspond to the actions of the first and second factors of the semidirect product $S_k\wr H$.  Combining these two facts, we arrive at the conclusion of the lemma.  
     \end{proof}  

The monodromy action on a fibre of a branched cover $\pi:\pi^{-1}(\cU) \to Y$ is a local action in the sense that it is usually not the restriction of a \mydefit{deck transformation}, which is a map $\phi:\pi^{-1}(\cU) \to \pi^{-1}(\cU)$  satisfying $\pi(\phi(x)) = \pi(x)$ for all $x \in \pi^{-1}(\cU)$.  Indeed, non-trivial deck transformations cannot fix any points. However, when the assumptions of \Cref{lem:blockcomponent} are met, we can identify many deck transformations of $X_{\mathcal B}^{(k)}\cap \left(\pi^{(k)}\right)^{-1}(\cU^{(k)})$.

\begin{lemma}
    \label{lem:Sk-deck_transformations}
    Let $\pi:X\rightarrow Y$ be a branched cover of degree $d$, and assume that $G_{\pi} \lneq S_d$ contains a transposition, then the action of $S_k$ on $X_{\mathcal B}^{(k)}$ defined by 
    \[
    \sigma((x_{i_1},\ldots,x_{i_k};y)) = (x_{i_{\sigma(1)}},\ldots,x_{i_{\sigma(k)}};y)
    \]is a deck-transformation when restricted to $X_{\mathcal B}^{(k)}\cap \left(\pi^{(k)}\right)^{-1}(\cU^{(k)})$. 
\end{lemma}

\begin{proof}
    If $\sigma \in S_k$ then $\pi^{(k)}(\sigma(x_I)) = \pi^{(k)}(x_I)$, since $\sigma$ only affects the ordering of the elements.  Let $u\in X_{\mathcal B}^{(k)}\cap \left(\pi^{(k)}\right)^{-1}(\cU^{(k)})$.  By \Cref{lem:blockcomponent}, $X_{\mathcal B}^{(k)}$ contains all orderings of all minimal blocks of $G_{\pi,u}$, so $\sigma(x_I) \in X_{\mathcal B}^{(k)}$.  Hence, $\sigma:X_\cB^{(k)}\rightarrow X_{\cB}^{(k)}$ and commutes with $\pi^{(k)}$.
\end{proof}

\begin{definition}
    Suppose $\pi:X \to Y$ is a branched cover of degree $d$ and $G_{\pi} \lneq S_d$ contains a transposition. We define $\mydef{X_{\mathcal B}^{\{k\}}}$ to be the quotient of $X_{\mathcal B}^{(k)}$ by the action of $S_k$ identified in \Cref{lem:Sk-deck_transformations}. 
\end{definition}

The quotient $X_{\mathcal B}^{\{k\}}$ is a variety since it is the orbit space of a finite group action on an irreducible variety \cite[Chapter 10]{Har92}.  Its points correspond to (unordered) $k$-sets instead of ordered $k$-tuples. In fact, by \Cref{cor:blocks} and \Cref{lem:blockcomponent}, its points are exactly the minimal blocks of $\pi: X\to Y$. 

We define two maps related to $X_{\mathcal B}^{\{k\}}$.  For the first map, we define $\mydef{\pi^{\{k\}}}:X_{\mathcal{B}}^{\{k\}}\rightarrow Y$ as follows.  For any ordered tuple $x_I$ of size $k$ in a fibre of $\pi$, we define $\mydef{x_I^{\{k\}}}\in X_{\mathcal B}^{\{k\}}$ to be the point which represents the underlying set of $x_I$.  In addition, we define $\pi^{\{k\}}(x_I^{\{k\}})$ as $\pi(x)$ for any $x\in x_I$.  Since $S_k$ acts on fibres of $\pi^{(k)}$, this map is well-defined.  Then $\pi^{\{k\}}$ is a degree-$b$ map since the fibre over $u \in \mathcal U^{(k)}$ consists of the $b$ blocks of $G_{\pi,u}$.  The second map is $\varphi_1:\pi^{-1}(\cU^{(k)}) \to X_{\mathcal B}^{\{k\}}\cap \left(\pi^{\{k\}}\right)^{-1}(\cU^{(k)})$ which takes $x \in \pi^{-1}(u)$ to the block of $G_{\pi,u}$ containing $x$. The preimage of a block $B$ under $\varphi_1$ is the set of elements  of $B$ and so $\varphi_1$ is a degree-$k$ map. 

\begin{theorem}\label{thm:explicitfactorization}
    Suppose $G_{\pi}\lneq S_d$ contains a transposition. Then $\pi$ decomposes as 
    \begin{center}
        \begin{tikzcd}
  X \arrow[r,longdashed,"\varphi_1"] \arrow[rr,longdashed,bend right=25,"\pi" below] & X^{\{k\}}_{\cB} \arrow[r,longdashed,"\pi^{\{k\}}"] & Y.
\end{tikzcd}
\end{center}
\end{theorem}
\begin{proof}
    Since $\pi:X \to Y$ is a branched cover, $G_{\pi}$ is transitive. By \Cref{cor:uniformorimprimitive}, since $G_{\pi}$ contains a transposition, it has the form $S_k \wr H$, that is, an imprimitive group with the $b=d/k$ orbits of $T(G_{\pi})$ as a unique minimal block system. Since $G_{\pi} \neq S_d$ we have $1 <k <d$. 

    From the explicit descriptions of $\varphi_1$ and $\pi^{\{k\}}$, we see that $\pi = \pi^{\{k\}} \circ \varphi_1$.  In addition, the degrees of $\varphi_1$ and $\pi^{\{k\}}$ are $k$ and $b$, respectively.  Since $1<k<d$, this is a decomposition of $\pi$. 
\end{proof}

We end with the statements that complete the final details of \Cref{prop:generalizedtracetest}.  In particular, we show that being trace-invariant is an algebraic condition.

\begin{lemma}\label{lem:coincidencealgebraicsubset}
Let $y^\ast\in\cU^{(d)}$ be generic and $\emptyset\not=P\varsubsetneq X_{y^\ast}$.  Then
$$
TI_P\vcentcolon=\{y:\exists [\tau]\in\Pi_1(\cU,y^\ast,y)\text{ s.t. }\sigma_\tau(P)\text{ is trace-invariant}\}
$$
is a Zariski closed subset of $\cU$.
\end{lemma}
\begin{proof}
Since $y^\ast\in\cU^{(d)}$, $X_{y^\ast}$ consists of $d$ distinct points. Suppose that the points have been labeled so that $X_{y^\ast}=\{x^\ast_1,\dots,x^\ast_d\}$ and $P=\{x^\ast_1,\dots,x^\ast_k\}$. Let $X_P$ be the component of $X^{(d)}$ containing $(x^\ast_1,\dots,x^\ast_d;y^\ast)$.  By \Cref{lem:groupoid_on_fibre_power}, $\pi_P:X_P\rightarrow Y$ is a branched cover.  For each $[\gamma]\in\Pi_1(\cU^{(d)},y^\ast)$, by a slight abuse of notation, we write $(x^\ast_{\sigma_\gamma(1)},\dots,x^\ast_{\sigma_\gamma(d)};y^\ast)\vcentcolon=(\sigma_\gamma(x_1^\ast),\dots\sigma_\gamma(x_d^\ast);y^\ast)$ to stress the permutation on the indices induced by $\gamma$.

We now consider the system of equations given by
\begin{equation}\label{eq:traceinvariantvariety}
\sum_{i=1}^kx_i=\sum_{i=1}^kx_{\sigma_\gamma(i)}\text{ for }[\gamma]\in\Pi_1(\cU^{(d)},y^\ast).
\end{equation}
We note that it is enough to generate these equations using a single representative $\gamma$ for each element of the monodromy group $G_{\pi,y^\ast}$.  Hence, this system of equations consists of finitely many equations, and they cut out a closed subvariety $C$ of $X_P$.  

Suppose that $(x'_1,\dots,x'_d;y')\in X_P$, then there is some $[\tau]\in\Pi_1(\cU^{(d)},y^\ast,y')$ such that $\sigma_\tau(x^\ast_i)=x'_i$.  Let $[\gamma']\in\Pi_1(\cU^{(d)},y')$, then there is some $[\gamma]\in\Pi_1(\cU^{(d)},y^\ast)$ so that $[\tau^{-1}\sqcdot\gamma\sqcdot\tau]=[\gamma']$.  Since homotopy equivalent paths induce the same bijections between their endpoints, we have that
$$
x'_{\sigma_{\gamma'}(i)}=\sigma_{\gamma'}(x'_i)=\sigma_\tau(\sigma_\gamma(\sigma_{\tau^{-1}}(x'_i)))=
\sigma_\tau(\sigma_\gamma(x^\ast_i))=\sigma_\tau(x^\ast_{\sigma_\gamma(i)})=x'_{\sigma_\gamma(i)}.
$$
In other words, the permutation induced by $\gamma'$ acts on the indices of $x'_i$ in the same way that $\gamma$ acts on the indices of $x^\ast_i$.

We next show that $(x'_1,\dots,x'_d;y')\in X_P$ satisfies \Cref{eq:traceinvariantvariety} if and only if $\{x'_1,\dots,x'_k\}$ is trace-invariant.  Let $[\gamma']\in\Pi_1(\cU^{(d)},y')$ and define $[\gamma]\in\Pi_1(\cU^{(d)},y^\ast)$ as above.  If the equations are satisfied, then 
$$
\sum_{i=1}^kx'_i=\sum_{i=1}^kx'_{\sigma_\gamma(i)}=\sum_{i=1}^k\sigma_{\gamma'}(x'_i),
$$
where the first equality represents satisfying the equations and the second equality uses the relationship between $\gamma$ and $\gamma'$.  Since this can be repeated for any $\gamma'$, this shows that $\{x'_1,\dots,x'_k\}$ is trace-invariant.  On the other hand, if $\{x'_1,\dots,x'_k\}$ is trace-invariant, then a similar equality holds,
$$
\sum_{i=1}^kx'_i=\sum_{i=1}^k\sigma_{\gamma'}(x'_i)=\sum_{i=1}^kx'_{\sigma_\gamma(i)},
$$
where the first equality represents the trace invariance and the second equality uses the relationship between $\gamma$ and $\gamma'$.  By direct observation, we see that the outer equality is one of those appearing in \Cref{eq:traceinvariantvariety}.  Moreover, since conjugation by $\tau$ gives a bijection between $\Pi_1(\cU^{(d)},y^\ast)$ and $\Pi_1(\cU^{(d)},y')$, we conclude that the entire system in \Cref{eq:traceinvariantvariety} is satisfied.

By taking closures, we see that the properties above can be extended from $C\cap\pi_P^{-1}(\cU^{(d)})$ to $C\cap\pi_P^{-1}(\cU)$.  Since $\pi_P:X_P\cap\pi_P^{-1}(\cU)\rightarrow\cU$ is a finite and dominant quotient, the image of $C\cap\pi_P^{-1}(\cU)$ is a closed subvariety of $\cU$.  Since the image of $C\cap\pi_P^{-1}(\cU)$ consists of those points $y\in \cU^{(d)}$ where \Cref{eq:traceinvariantvariety} holds for some point in the fibre over $C$, this image is the variety of interest.
\end{proof}

\begin{corollary}\label{cor:entirefibres}
Under the hypotheses and notation of \Cref{lem:coincidencealgebraicsubset} and its proof, $C\cap\pi_P^{-1}(\cU)=\cU$ if and only if
$$\{y:\exists [\tau]\in\Pi_1(\cU,y^\ast,y)\text{ s.t. }\sigma_\tau(P)\text{ is trace-invariant}\}=\cU.$$
\end{corollary}
\begin{proof}
If $C\cap\pi_P^{-1}(\cU)=\cU$, then $C\cap\pi^{-1}_P(\cU)$ contains at least one point in every fibre of $\pi_P:X_P\cap \pi_P^{-1}(\cU)\rightarrow\cU$.  Hence, the image under this map is $\cU$.  On the other hand, suppose that $[\tau_1],[\tau_2]\in\Pi_1(\cU,y^\ast,y')$ and $\sigma_{\tau_1}(P)$ is trace-invariant. Let $[\gamma']\in\Pi_1(\cU,y')$.  Since $\tau_2\sqcdot\gamma'$ is a path from $y^\ast$ to $y'$, we conclude from the proof of \Cref{prop:generalizedtracetest} that $\Sigma(\sigma_{\gamma'}(\sigma_{\tau_2}(P)))=\Sigma(\sigma_{\tau_1}(P))$.  Since this is independent of $\gamma'$, we conclude that $\sigma_{\tau_2}(P)$ is trace-invariant.

This implies that $C\cap\pi_P^{-1}(\cU)$ consists of entire fibres of $\pi_P$.  Hence, if the projection is surjective, then $C\cap\pi_P^{-1}(\cU)$ must have a point in every fibre of $\pi_P$.  Since the intersection of $C\cap\pi_P^{-1}(\cU)$ with a fibre is either empty or the entire fibre and the intersections are not empty, we conclude that $C\cap\pi_P^{-1}(\cU)$ includes every fibre, and the equality holds.
\end{proof}

\section*{AI Statement} In this paper, AI was used for drafting, editing, proof exposition, literature navigation, and mathematical brainstorming. Most of the AI use was in writing, exposition, and development of figures, with a smaller but meaningful role in mathematical reasoning and reference checking.  The statement and proof of \Cref{lem:sufficient_for_no_TIS} were initially suggested by AI.  The research questions, central results, and final mathematical judgments were developed and verified by the authors.

\bibliographystyle{amsplain}
\bibliography{references}

\end{document}